\documentclass[11pt,a4paper]{amsart}

\usepackage{amsmath,amsthm,amssymb}
\usepackage{tikz}
\usepackage{caption,verbatim} 
\usetikzlibrary{arrows,decorations.pathmorphing,decorations.pathreplacing}

\tikzset{ >=stealth',
         leadsto/.style={-angle 90,decorate,decoration=snake,very thick},
         source/.style={inner sep=1pt,circle,draw,thick},
         sink/.style={inner sep=2pt,rectangle,draw,thick}
}

\newtheorem{theorem}{Theorem}[section]

\newtheorem{algorithm}[theorem]{Algorithm}

\newtheorem{corollary}[theorem]{Corollary}

\newtheorem{lemma}[theorem]{Lemma}
\newtheorem{proposition}[theorem]{Proposition}
\theoremstyle{definition}
\newtheorem{definition}[theorem]{Definition}

\theoremstyle{definition}
 
\theoremstyle{definition}

\theoremstyle{definition}
\newtheorem{example}[theorem]{Example}

\theoremstyle{definition}

\theoremstyle{remark}
\newtheorem*{remark}{Remark}
\newtheorem*{remarks}{Remarks}
\theoremstyle{definition}

\theoremstyle{definition}

\DeclareMathOperator{\add}{add}

\newcommand{\can}{\operatorname{can}\nolimits}
\DeclareMathOperator{\coh}{coh}
\newcommand{\cohX}{\coh \mathbb{X}}
\DeclareMathOperator{\Coker}{Coker}
\DeclareMathOperator{\Ext}{Ext}
\DeclareMathOperator{\End}{End}

\DeclareMathOperator{\Hom}{Hom}

\DeclareMathOperator{\Ker}{Ker}
\renewcommand{\mod}{\operatorname{mod}\nolimits}

\DeclareMathOperator{\rad}{rad}
\DeclareMathOperator{\rank}{rk}
\DeclareMathOperator{\pd}{pd}
\DeclareMathOperator{\inj}{inj}
\newcommand{\nsource}{\ensuremath{\operatorname{source}\nolimits^-} }
\newcommand{\nsink}{\ensuremath{\operatorname{sink}\nolimits^-} }

\newcommand{\sq}{\operatorname{sq}\nolimits}

\newcommand{\vect}{\operatorname{vect}\nolimits}
\DeclareMathOperator{\Cone}{Cone}
\newcommand{\D}{\operatorname{D}\nolimits \!}
\newcommand{\R}{\operatorname{R}\nolimits \!}
\renewcommand{\L}{\operatorname{L}\nolimits \!}

\newcommand{\bfE}{\mathbf{E}}
\newcommand{\bfM}{\mathbf{M}}
\newcommand{\bfX}{\mathbf{X}}
\newcommand{\sep}{\!\! : \!\!}

\newcommand{\extto}{\xrightarrow}

\newcommand{\from}{\colon \!}

\newcommand{\into}{\rightarrowtail}

\newcommand{\lambdab}{ \text{\boldmath{$\lambda$}}}

\newcommand{\lperp}[1]{\sideset{^\perp}{}{\operatorname{#1}}}

\newcommand{\onto}{\twoheadrightarrow}

\newcommand{\rperp}[1]{\sideset{}{^\perp}{\operatorname{#1}}}

\newcommand{\bfL}{ \mathbf{L}}  
\newcommand{\bfp}{ \mathbf{p}}  

\newcommand{\bbK}{ \mathbb{K}}  
\newcommand{\bbN}{ \mathbb{N}}  
\newcommand{\bbP}{ \mathbb{P}}  
\newcommand{\bbX}{ \mathbb{X}}  
\newcommand{\bbZ}{ \mathbb{Z}}  

\newcommand{\cC}{{\mathcal C}} 
\newcommand{\cD}{{\mathcal D}} 
 
\newcommand{\cE}{{\mathcal E}} 
\newcommand{\cF}{{\mathcal F}} 
\newcommand{\cH}{{\mathcal H}} 
\newcommand{\cI}{{\mathcal I}} 
\newcommand{\cL}{{\mathcal L}} 
\newcommand{\cR}{{\mathcal R}} 
\newcommand{\cO}{{\mathcal O}} 
\newcommand{\cP}{{\mathcal P}} 
\newcommand{\cS}{{\mathcal S}} 
\newcommand{\cT}{{\mathcal T}} 
\newcommand{\cX}{{\mathcal X}} 
\newcommand{\cY}{{\mathcal Y}} 

\newcommand{\vc}{ {\vec{c}}} 
\newcommand{\vx}{ {\vec{x}}}

\newcommand{\tcL}{\widetilde{\cL}}
\newcommand{\lambdat}{\tilde{\lambda}}
\newcommand{\tcR}{\widetilde{\cR}}
\newcommand{\rhot}{\tilde{\rho}}

\numberwithin{figure}{section}
\numberwithin{equation}{section}

\title[Graded mutation]{
Graded mutation in cluster categories coming from hereditary categories with a tilting object}

\author[Bertani-{\O}kland]{Marco Angel Bertani-{\O}kland}
\author[Oppermann]{Steffen Oppermann}

\address{Institutt for matematiske fag\\ NTNU\\ 7491 Trondheim\\ Norway}
\email{Marco.Tepetla@math.ntnu.no}
\email{Steffen.Oppermann@math.ntnu.no}

\author[Wr{\aa}lsen]{Anette Wr{\aa}lsen}
\address{Faculty of Informatics and e-learning\\ S{\o}r-Tr{\o}ndelag University College\\ 7004 Trondheim\\ Norway}
\email{anette.wralsen@hist.no}

\begin{document}

\begin{abstract}
We present a graded mutation rule for quivers of cluster-tilted algebras. Furthermore, we give a technique to recover a cluster-tilting object from its graded quiver in the cluster category of $\cohX$.
\end{abstract}

\maketitle

\section{Introduction}

 Let $\bbK$ be an algebraically closed field and $\cH$ be a connected hereditary $\bbK$-category with a tilting object. It was shown in \cite{Happel2} that $\cH$ is derived equivalent to $\mod H$ for some finite dimensional hereditary $\bbK$-algebra $H$, or to the category $\cohX$ of coherent sheaves on a weighted projective line $\bbX$. The cluster category $\cC:=\cC_{\cH}$, an orbit category of the derived category of $\cH$, was introduced in \cite{BMRRT}. If $\cH$ has no nonzero projectives then the cluster category comes with a natural grading on morphisms, inherited from $\cH$ in the sense that any map in $\cC$ can be written as a sum of a morphism and an extension in $\cH$. 

One important reason for studying cluster categories is that they
contain certain special objects: Cluster-tilting objects and their endomorphism rings, the so-called cluster-tilted algebras (see Section~\ref{section.cluster.cats}). These algebras have been extensively investigated, see for instance \cite{BMR1,BMR2,BR,BRS,CCS,ABS}. 

In this paper we introduce a graded mutation rule for the cluster
categories coming from hereditary categories $\cH$ with a tilting
object. Graded mutation is a way to mutate preserving the natural
grading of a cluster-tilted algebra (see
Section~\ref{section.gradmut}). This rule extends the quiver mutation
rule of Fomin-Zelevinsky, and is an adaptation of the mutation rule
for tilting sheaves in $\cohX$ given in \cite{Hubner1} to the
cluster-tilting case. In order to introduce this rule, we define each
indecomposable summand of a cluster-tilting object to be either a sink
or a source. The sink/source property was defined originally for
$\cohX$ in \cite{Hubner1}, but we present a way to do this in any
$\cH$. Then we lift this property to $\cC$ using the natural
correspondence between tilting objects in $\cH$ and cluster-tilting
objects in $\cC$ (see \cite[Proposition~3.4]{BMRRT}).
  
The main result of this article is a positive answer to the following recovery problem: Let $T$ be a basic cluster-tilting object in $\cC_{\bbX}$, the cluster category of $\cohX$ for some weighted projective line $\bbX := (\bbP_{\bbK}^1,\lambdab, \bfp)$, with cluster-tilted algebra $\Gamma = \End_{\cC_{\bbX}}(T)$. If we are given only the graded quiver $Q_{\Gamma}$ and the rank of the indecomposable summands of $T$, can we recover $T$ and $\Gamma$? 

In order to present a solution, we define the notion of the quiver of an exceptional sequence and develop a mutation rule for these quivers. Then we provide a concrete algorithm that transforms $T$ into a well-known cluster-tilting object in $\cC_{\bbX}$ by applying successive mutations of quivers of exceptional sequences (see Theorem~\ref{theorem.reconstruct_from_Q+rk}). Keeping track of the mutations involved, we are able to determine $T$ in $\cohX$ up to twist with a line bundle and choice of parameter sequence $\lambdab$.

The paper is organized as follows: 

In Section~\ref{section.background} we recall some basic results on hereditary categories with a tilting object, cluster categories, quiver mutation and exceptional sequences. 

 For coherent sheaves on a weighted projective line, the notion of
 being a sink or a source of an indecomposable summand of a tilting object has been introduced by H{\"u}bner. In Section~\ref{section.sink.source} we extend this notion to arbitrary hereditary categories.

In Section~\ref{section.gradmut} we adapt the sink/source property to indecomposable summands of cluster-tilting objects by choosing a ``canonical'' hereditary category $\cH_*$ from which we construct the cluster category. This category $\cH_*$ will also allow us to define a {\em natural grading} of a cluster-tilted algebra. We then present our graded mutation rule, and explore its relations with the sink-source distribution of a cluster-tilting object.

Section~\ref{section.mut_excep_sequences} develops the theory of mutation of quivers of exceptional sequences. 

In Section~\ref{section.recoverT} we give the algorithm to recover
$T$ from certain given combinatorial data.

Finally, in Section~\ref{lastsection} we collect some natural
questions on what information the quiver of a cluster-tilted algebra
contains, and to what extent answers are already known or are obtained
within this paper.                                                       

\subsection*{Acknowledgements}

Research for this paper began during a visit of the first and third authors to the University of Paderborn in 2009. We would like to thank Helmut Lenzing for inspirational suggestions and many helpful conversations. We would also like to thank the Paderborn representation theory group for their hospitality. 

\section{Background} \label{section.background}

\subsection{Hereditary categories with a tilting object}

Let $\mathcal H$ be a connected hereditary abelian category over an algebraically closed field $\bbK$, and assume that $\cH$ is $\Hom$-finite. Furthermore assume that $\mathcal H$ has a \emph{tilting object}, that is, an object $T$ such that $\Ext^1_{\mathcal H}(T,T) = 0$ and such that the number of nonisomorphic indecomposable summands of $T$ is the same as the rank of $K_0(\mathcal H)$. Then it is shown in \cite{Happel2} that $\mathcal H$ is derived equivalent to $\mod H$ for some finite dimensional hereditary $\bbK$-algebra H, or derived equivalent to $\coh \mathbb X$, the category of coherent sheaves on the weighted projective line  $\bbX = (\mathbb P^1_{\bbK},\lambdab, \bfp)$ for some sequence of pairwise distinct points $\lambdab$ and some weight sequence $\bfp$. Note that $\mathcal H$ has almost split sequences (\cite{HRS}) and thus an Auslander-Reiten quiver (see \cite{ARS} for more background).

In particular it is known (see \cite[Theorem~4.2]{Happel1}) that if $\mathcal H$ has nonzero projectives, it is equivalent to the  
category of finitely generated modules over some finite dimensional hereditary $\bbK$-algebra $H$. Any such algebra is Morita  
equivalent to the path algebra $\bbK Q$ for some finite quiver $Q$. The quivers for which the algebra is of finite or tame representation type are explicitly known. They are the quivers of simply-laced \emph{Dynkin type} or \emph{Euclidean type}, respectively. For a nice survey on this subject we refer the reader to \cite{Lenzing2}. If the algebra is representation finite, the Auslander-Reiten quiver of $D^b(H)$ is of the form $\mathbb ZQ$, if it is tame it has components of type $\mathbb ZQ$ as well as tubes of finite rank, and if it is wild the Auslander-Reiten quiver consists of $\mathbb ZQ$ and $\mathbb ZA_{\infty}$ components.

If $\mathcal H$ is not derived equivalent to a hereditary category with nonzero projectives, then $D^b(\mathcal H)$ is equivalent to $D^b(\coh \mathbb X)$ for some weighted projective line $\mathbb X$. These categories also come in three classes with respect to their representation type, namely \emph{tame domestic}, \emph{tubular} and \emph{wild}. In the first case $D^b(\coh \mathbb X)$ is equivalent to $D^b(\mod H)$ for some tame hereditary algebra $H$, and for any tame hereditary algebra $H$ there is such an equivalence. This is the only case in which these two types of categories intersect (see Figure~\ref{figure.all_types}). The Auslander-Reiten quiver of $D^b(\mathcal H)$ in this case thus combines $\mathbb ZQ$ components and tubes. In the tubular case it only consists of tubes and in the wild case the Auslander-Reiten quiver consists of tubes and $\mathbb ZA_{\infty}$ components (see \cite{LP}).

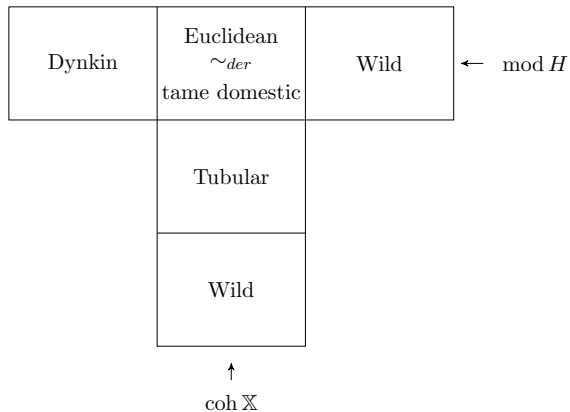
\begin{figure}
\[\scalebox{0.75}{
\begin{tikzpicture}[scale=1,xscale=1.3,yscale=-1]
\draw (0,0) -- (6,0) -- (6,2) -- (0,2) -- (0,0);
\draw (2,0) -- (2,6) -- (4,6) -- (4,0);
\draw (2,4) -- (4,4);
\node at (1,1) {Dynkin};
\node at (3,0.5) {Euclidean};
\node at (3,1) {$\sim_{der}$};
\node at (3,1.5) {tame domestic};
\node at (5,1) {Wild};
\node at (7.1,1) {$\mod H$};
\node at (3,3) {Tubular};
\node at (3,5) {Wild};
\node at (3,7) {$\cohX$};
\draw [->] (6.4,1) -- (6.1,1);
\draw [->] (3,6.6) -- (3,6.3);
\end{tikzpicture}}
\]
\caption{ Hereditary categories with a tilting object. }  \label{figure.all_types}
\end{figure}

\subsubsection{The category of coherent sheaves on a weighted projective line}

In this subsection we briefly recall some properties of the category of coherent sheaves on a weighted projective line. The purpose is mainly to fix notation. We refer the reader to \cite{GL, Lenzing1, Lenzing2,CK} for further background.

Let $\bbK$ be an algebraically closed field. A {\em parameter sequence} is a (possibly empty) sequence $\lambdab=(\lambda_1,\ldots,\lambda_t)$ of pairwise distinct points of the projective line $\bbP^1_\bbK$. A sequence $\bfp=(p_1,\ldots,p_t)$ of integers greater than 1 is called a {\em weight sequence}. Then a {\em weighted projective line} is a triple $\bbX=(\bbP^1_\bbK,\lambdab,\bfp)$, where $\lambdab$ and $\bfp$ are respectively a parameter sequence and a weight sequence of the same length. The category $\cohX$ of coherent sheaves on the weighted projective line $\mathbb X$ is defined as follows. We denote by $\bfL (\bfp)$ the rank 1 abelian group
\[
 \bfL( \bfp ) = \langle \vx_1,\ldots,\vx_t,\vc \, \mid \, p_1\vx_1 = \cdots = p_t\vx_t=\vc \rangle.
\]
This is an ordered group with $\bfL_+=\sum_{i=1}^t \bbN \vx_i$ as its set of positive elements. Then the algebra
\[
 S(\bfp,\lambdab) = \bbK[u,v,x_1,\ldots,x_t]/(x_i^{p_i} - \lambda_i^0 u - \lambda_i^1 v),
\]
where $\lambda_i=[\lambda_i^0 \sep \lambda_i^1]\in \bbP^1_\bbK$, becomes an $\bfL(\bfp)$-graded algebra by defining $\deg u=\deg v=\vc$ and $\deg x_i=\vx_i$. The category $\cohX$ is defined as the quotient of the category of finitely generated $\bfL(\bfp)$-graded $S(\bfp,\lambdab)$-modules modulo the Serre subcategory of finite length modules. Recall that a {Serre subcategory} is a full subcategory closed under taking subobjects, quotients and extensions.

In $\cohX$ we have the two disjoint subcategories $\coh_0 \bbX$ and $\vect \bbX$, where the first are all the objects of finite length and the latter are all the objects without simple subobjects. The objects of $\coh_0 \bbX$ are called \emph{torsion sheaves} while the objects of $\vect \bbX$ are called \emph{torsion-free sheaves} or {\em vector bundles}. By \cite[Proposition 10.1]{Lenzing2} we know that each indecomposable object from $\cohX$ either belongs to $\coh_0 \bbX$ or to $\vect \bbX$.

An important property of $\cohX$ is that its Grothendieck group $K_0 \bbX$ is finitely generated free. There is a useful linear map from $K_0 \bbX$ to $\mathbb Z$ called the \emph{rank}. This function has the property that $\rank \tau X = \rank X$ for all $X \in \cohX$, and $\rank X = 0$ for all $X \in \coh_0 \bbX$ while $\rank X > 0$ for all nonzero $X \in \vect \bbX$. Moreover we know that there exists a vector bundle of rank one, the \emph{structure sheaf $\cO$}, induced by the free module $S(\bfp,\lambdab)$. Vector bundles of rank one are also referred to as \emph{line bundles}. Moreover, every line bundle is of the form  $\cO(\vx)$ for some uniquely determined $\vx \in \bfL(\bfp)$. For proofs of all these properties we refer to \cite{Lenzing1,CK}.

One special property of the rank is \emph{additivity on tilting sheaves} (a tilting sheaf is a tilting object in $\cohX$). 

\begin{definition}
Let $T = T_1 \oplus T_2 \oplus \cdots \oplus T_n$ be a tilting sheaf, where the $T_i$ are indecomposable for all $1 \leq i \leq n$. We denote by $Q_T$ be the quiver (with relations) of $\End_{\cohX}(T)$. A linear function $f \colon K_0 \bbX \rightarrow \mathbb Z$ is \emph{additive} on $T$, if

\begin{itemize}
\item[(i)]  $f(T_i) \geq 0$ for all $1 \leq i \leq n$, and
\item[(ii)] $2 \cdot f(T_i) = \sum\limits_{j} a(j,i) \cdot f(T_j)  
- \sum\limits_{j} b(j,i) \cdot f(T_j)$
\end{itemize}
where $a(i,j)$ denotes the number of arrows between the vertices in $Q_T$ corresponding to $T_i$ and $T_j$ (in either direction, $a(i,j)=0$ if there are no such arrows), and similarly $b(i,j)$ denotes the number of relations between the vertices corresponding to $T_i$ and $T_j$.
\end{definition}

Then we have the following (\cite[Theorem~3.2]{Hubner3}):

\begin{theorem} \label{rank-is-additive}
The rank function is additive on each tilting sheaf.
\end{theorem}

This phenomenon is special to the case of tilting objects on a weighted projective line, and plays a crucial role in their study. It will be made use of throughout this paper.

There are two special tilting sheaves in $\coh \mathbb X$ which will appear later in this paper. 

\begin{definition}

The {\em canonical} tilting sheaf (starting in the line bundle $\cO$)
is given by $T_{\can}=\oplus_{\vx} \cO(\vx)$ where $0\le \vx \le
\vc$. Its endomorphism ring is called a {\em canonical algebra} and is
given by the quiver
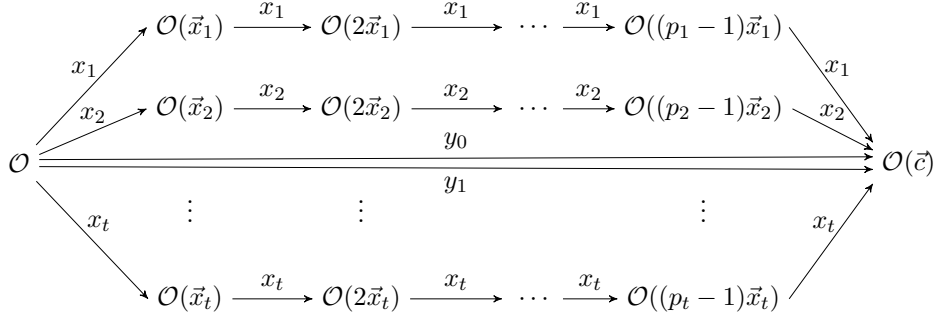
\[\scalebox{0.9}{
\begin{tikzpicture}[scale=2.5,yscale = -0.8]
\node (ox1) at (2,1) {$\cO(\vx_1)$};
\node (o2x1) at (3,1) {$\cO(2\vx_1)$};
\node (cdots1) at (4,1) {\phantom{.}$\cdots$};
\node (op1x1) at (5,1) {$\cO((p_1-1)\vx_1)$};
\node (o) at (1,2) {$\cO$};
\node (ox2) at (2,1.6) {$\cO(\vx_2)$};
\node (o2x2) at (3,1.6) {$\cO(2\vx_2)$};
\node (cdots2) at (4,1.6) {\phantom{.}$\cdots$};
\node (op2x2) at (5,1.6) {$\cO((p_2-1)\vx_2)$};
\node (oc) at (6.2,2) {$\cO(\vc)$};
\node (vdots1) at (2,2.3) {$\vdots$};
\node (vdots2) at (3,2.3) {$\vdots$};
\node (vdots4) at (5,2.3) {$\vdots$};
\node (oxt) at (2,3) {$\cO(\vx_t)$};
\node (o2xt) at (3,3) {$\cO(2\vx_t)$};
\node (cdotst) at (4,3) {\phantom{.}$\cdots$};
\node (optxt) at (5,3) {$\cO((p_t-1)\vx_t)$};
\draw [->] (o) -- node [above,xshift=-1mm] {$x_1$} (ox1.west);
\draw [->] (ox1) -- node [above] {$x_1$} (o2x1);
\draw [->] (o2x1) -- node [above] {$x_1$} (cdots1);
\draw [->] (cdots1) -- node [above] {$x_1$} (op1x1);
\draw [->] (o) -- node [above] {$x_2$} (ox2.west);
\draw [->] (ox2) -- node [above] {$x_2$} (o2x2);
\draw [->] (o2x2) -- node [above] {$x_2$} (cdots2);
\draw [->] (cdots2) -- node [above] {$x_2$} (op2x2);
\draw [->] (op1x1)+(0.5,0) -- node [above,xshift=1mm] {$x_1$} (oc.150);
\draw [->] (op2x2.east) -- node [above] {$x_2$} (oc.160);
\draw [->] (optxt)+(0.5,0) -- node [above,xshift=-1mm] {$x_t$} (oc.210) ;
\draw [->] (o) -- node [above,xshift=1mm] {$x_t$} (oxt.west);
\draw [->] (oxt) -- node [above] {$x_t$} (o2xt);
\draw [->] (o2xt) -- node [above] {$x_t$} (cdotst);
\draw [->] (cdotst) -- node [above] {$x_t$} (optxt);
\draw [->] (o.10) -- node [above] {$y_0$} (oc.170);
\draw [->] (o.350) -- node [below] {$y_1$} (oc.190);
\end{tikzpicture}}
\]
with the relations $x_i^{p_i} = \lambda_i^0 y_0 +\lambda_i^1 y_1$ for
$i =1, \ldots, t$. (Note that for $t \geq 2$ the arrows $y_0$ and $y_1$ are generated by the other arrows, so they do not appear in the Gabriel quiver of the algebra. Similarly for $t=1$ there is only one $y$-arrow in the Gabriel quiver.)

\end{definition}

\begin{definition}\label{Squid} 
 The {\em squid} tilting sheaf (starting in the line bundle $\cO$) is given by $T_{\sq}=\cO\oplus \cO(\vc) \oplus_{i,j} S_i^{(j)} $ where $1 \le i \le t$ and $1\le j \le p_i -1$. Its endomorphism ring is called a {\em squid algebra} and has the following quiver:
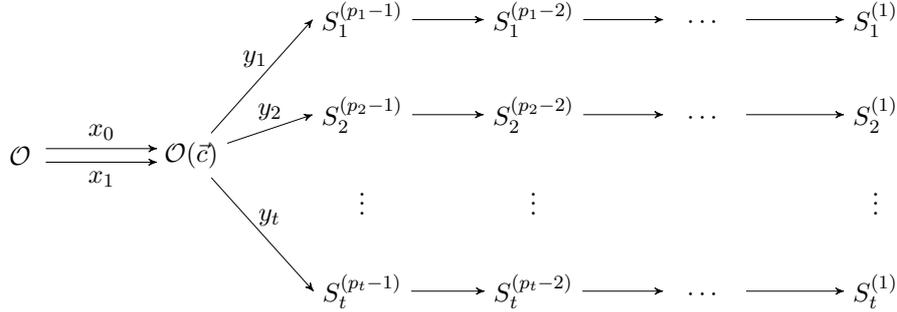
\[ \scalebox{0.9}{
\begin{tikzpicture}[scale=2.5,yscale = -0.8]
\node (s1p-1) at (3,1) {$S_1^{(p_1-1)}$};
\node (s1p-2) at (4,1) {$S_1^{(p_1-2)}$};
\node (cdots1) at (5,1) {\phantom{S}$\cdots$\phantom{S}};
\node (s11) at (6,1) {$S_1^{(1)}$};
\node (o) at (1,2) {$\cO$};
\node (oc) at (2,2) {$\cO(\vc)$};
\node (s2p-1) at (3,1.7) {$S_2^{(p_2-1)}$};
\node (s2p-2) at (4,1.7) {$S_2^{(p_2-2)}$};
\node (cdots2) at (5,1.7) {\phantom{S}$\cdots$\phantom{S}};
\node (s21) at (6,1.7) {$S_2^{(1)}$};
\node (vdots1) at (3,2.3) {$\vdots$};
\node (vdots2) at (4,2.3) {$\vdots$};
\node (vdots4) at (6,2.3) {$\vdots$};
\node (stp-1) at (3,3) {$S_t^{(p_t-1)}$};
\node (stp-2) at (4,3) {$S_t^{(p_t-2)}$};
\node (cdots4) at (5,3) {\phantom{S}$\cdots$\phantom{S}};
\node (st1) at (6,3) {$S_t^{(1)}$};
\draw [->] (s1p-1) -- (s1p-2);
\draw [->] (s1p-2) -- (cdots1);
\draw [->] (cdots1) -- (s11);
\draw [->] (1.15,1.95) -- node [above] {$x_0$} (1.8,1.95);
\draw [->] (1.15,2.05) -- node [below] {$x_1$}(1.8,2.05);
\draw [->] (oc) -- node [above,xshift=-1mm] {$y_1$}(s1p-1.west);
\draw [->] (oc) -- node [above] {$y_2$}(s2p-1.west);
\draw [->] (oc) -- node [above,xshift=1mm] {$y_t$}(stp-1.west);
\draw [->] (s2p-1) -- (s2p-2);
\draw [->] (s2p-2) -- (cdots2);
\draw [->] (cdots2) -- (s21);
\draw [->] (stp-1) -- (stp-2);
\draw [->] (stp-2) -- (cdots4);
\draw [->] (cdots4) -- (st1);
\end{tikzpicture}}
\]
Here the $S_i^{(1)}$ are the simple sheaves concentrated in $\lambda_i$, and for $t \in \{2, \ldots, p_i-1\}$ the $S_i^{(t)}$ are the sheaves of length $t$ with top $S_i^{(1)}$. Its relations are $(\lambda_i^0 x_0 + \lambda_i^1 x_1)y_i=0$ for $1\le i \le t$.
\end{definition}

\subsection{Cluster categories}\label{section.cluster.cats}

The cluster category $\mathcal C_H$ of a hereditary algebra $H$ was introduced in \cite{BMRRT}. It is defined as the orbit category $D^b(H)/F$, where $F$ is the functor $\tau^{-1}[1]$ composed of the inverse of the Auslander-Reiten translation $\tau$ and the shift functor $[1]$. This is a triangulated (\cite{Keller}) Krull-Schmidt (\cite{BMRRT}) category which is an important tool to study the tilting theory of $H$. It also models the combinatorics of the cluster algebras introduced by Fomin and Zelevinsky in 2002 (see \cite{FZ1}, \cite{FZ2}) in a natural way.

The construction of cluster categories extends naturally to any hereditary category $\mathcal H$ with a tilting object. For the remainder of this section we will write $\mathcal{C}$ and $\mathcal{D}$ for the cluster category and bounded derived category of any such category $\mathcal{H}$, respectively. In statements only referring to the categories coming from some hereditary algebra $H$, we write $\mathcal{C}_H$ and $\mathcal{D}_H$, respectively.

The objects of $\mathcal C$ are the objects in $\mathcal D$. Given objects $X$ and $Y$ in $ \mathcal D$, the space $\Hom_{\mathcal C}(X, Y)$ is  defined as 
\begin{equation}\label{eq.hom.spaces.cluster}
 \Hom_{\mathcal C}(X, Y) = \coprod_{i \in \mathbb Z} \Hom_{\mathcal D}(X,F^i Y).
\end{equation}

The  morphisms in $\mathcal C$ are thus induced by morphisms and extensions in $\mathcal H$.

This category has a set of objects with nice combinatorial properties, namely the \emph {cluster-tilting objects}. These are the maximal rigid objects of $\cC$, that is, the rigid objects that are maximal with respect to the number of nonisomorphic indecomposable summands. A cluster-tilting object is called {\em basic} if all of its indecomposable direct summands are nonisomorphic. All cluster-tilting objects occurring throughout this paper are assumed to be basic.

By \cite[Theorem~3.3]{BMRRT} the cluster-tilting objects of $ \mathcal C_H$ are all induced by tilting modules over a hereditary algebra $H'$ derived equivalent to $H$. This correspondence between tilting and cluster-tilting objects also extends naturally  to $\mathcal H = \coh \mathbb X$. If $T$ is a cluster-tilting object in $\mathcal C$, we refer to $\End_{\mathcal C}(T)$ as a \emph{cluster-tilted algebra}.

In some cases there exists a simpler construction of $\cC$. If $\cH$ has no nonzero projectives, then $\cH$ itself is a fundamental domain of $\cC$, that is, the canonical functor $\mathcal{D} \to \mathcal{C}$ induces a bijection between the isomorphism classes of objects in $\mathcal{H}$ and in $\mathcal{C}$. So let $\cH$ be without nonzero projectives. Then $\cC$ is equivalent to the category with the same objects as $\cH$, and with morphism spaces given by $\Hom_{\cC}(X,Y) = \Hom_{\cH}(X,Y) \oplus \Ext^1_{\cH}(X,\tau^{-1}Y)$. This decomposition comes from Equation~(\ref{eq.hom.spaces.cluster}), which induces a natural $\mathbb{Z}$-grading of the morphisms spaces. Here the morphisms in the first summand are of degree zero, and those in the second summand are of degree one. Note that if $X,Y\in \cH$ then $\Hom_{\mathcal D}(X,F^i Y)= \Ext^i_{\cH}(X,\tau^{-i} Y)$, so this space vanishes for $i \not\in \{0,1\}$. In particular any composition of degree one morphisms in $\cC$ vanishes.

The $\bbZ$-grading of the morphism spaces also induces a natural grading on the cluster-tilted algebras associated with $\cH$ and their quivers. This construction is studied for the coherent sheaves in \cite{BKL}, and it works the same way for any hereditary category without nonzero projectives.

\subsection{Quiver mutation}

Assume that $Q$ is a finite quiver with no loops and no
$2$-cycles. Also assume that the vertices of $Q$ are numbered from $1$
to $n$. To apply Fomin-Zelevinsky quiver mutation to vertex $k$ and
obtain $Q^* = \mu_k(Q)$, we do the following:

\begin{enumerate}

    \item  If there are $r \ge 0$ arrows $i \rightarrow k$, $s \ge 0 $ arrows $k \rightarrow j$ and $t$ arrows $j \rightarrow i$ in $Q$ (where a negative number of arrows means that the arrows go in the opposite direction of what is indicated), there are $r$ arrows $k \rightarrow i$, $s$ arrows $j \rightarrow k$ and $t-rs$ arrows $j \rightarrow i$ in $Q^*$.

    \item  All other arrows in $Q^*$ are the same as in $Q$.

\end{enumerate}

Note that $\mu_k^2(Q) = Q$. We say that $Q$ and $Q'$ lie in the same {\em mutation component} if there is a sequence of mutations taking $Q$ to $Q'$. The collection of all quivers lying in the mutation component of $Q$ is called the {\em mutation class} of $Q$.

\subsection{Exceptional sequences} \label{Section.Exceptional_sequences}
 This section deals with the necessary results on exceptional sequences. Further details can be found in \cite{CB,Meltzer,Ringel3}.

Let $\cH$ be a hereditary $\bbK$-category with a tilting object over an algebraically closed field $\bbK$. Denote by  $\cD:=D^b(\cH)$ its bounded derived category. Let $K_0(\cH)$ and $K_0(\cD)$ be their corresponding Grothendieck groups. It is a well know fact that these two groups are isomorphic. An object $E$ in $\cH$ (resp.\ $\cE$ in $\cD$) is called {\em exceptional} if it is indecomposable and $\Ext^1_\cH(E,E)=0$ (resp.\ $\Hom_\cD(\cE,\cE[i])$ for $i\ne 0$).

\begin{remark}
Let $\cE$ in $\cD$ be indecomposable. Then $\cE \simeq E[i]$ for some integer $i\in \bbZ$ and an indecomposable object $E$ in $\cH$. Throughout the rest of this paper, we denote elements of $\mathcal{H}$ by non-script letter ($E$, $F$, $\ldots$), and elements of $\mathcal{D}$ which do not necessarily lie in $\mathcal{H}$ by script letters ($\mathcal{E}$, $\mathcal{F}$, $\ldots$).
\end{remark}

\begin{definition}
 A sequence $(\cE_1,\ldots,\cE_n)$ of exceptional objects in $\cD$ is called an {\em exceptional sequence of length} $n$ if $\Hom_{\cD}(\cE_i,\cE_j[l])=0$ for all $l\in\bbZ$ whenever $1\le j<i\le n$. 
\end{definition}

By the previous remark, any exceptional sequence $(\cE_1,\ldots,\cE_n)$ is of the form $(E_1[l_1],\ldots,E_n[l_n])$ for suitable $l_1,\ldots,l_n$ in $\bbZ$ and exceptional objects $E_1,\ldots,E_n$ in $\cH$. If all $l_1,\ldots,l_n$ are zero, we call the sequence $(E_1,\ldots,E_n)$ an {\em exceptional sequence} in $\cH$. An exceptional sequence $(\cE_1,\ldots,\cE_n)$ is {\em complete} if $n= \rank K_0(\cH)$. Since we are only interested in isomorphism classes of objects, two exceptional sequences will be considered the same if the objects in the corresponding positions are isomorphic.

 An exceptional sequence of length $2$ will be called an {\em   exceptional pair}. Furthermore we call an exceptional pair $(E,F)$ {\em orthogonal} if we also have that $\Hom(E,F)=0$. For an orthogonal pair $(E,F)$, we may consider the category $\cC(E,F)$ of all objects in $\cH$ having a filtration with factors isomorphic to $E$ and $F$. By \cite[Section~1]{Ringel1}, this is an abelian subcategory of $\mathcal{H}$ with two simple objects, namely $E$ and $F$, which is equivalent to the category of modules {over the quiver algebra of the generalized Kronecker quiver
$\circ
\begin{tikzpicture}[baseline]
\node at (.3,.15) {$\cdot$};
\node at (.3,.09) {$\cdot$};
\node at (.3,.03) {$\cdot$};
\draw [->] (0,.2) -- (.6,.2);
\draw [->] (0,0) -- (.6,0);
\end{tikzpicture} \circ$
with the number of arrows being $\dim \Ext^1(E,F)$}. The only orthogonal pair in $\cC(E,F)$ is the one corresponding to the simple objects $E$ and $F$.

For an exceptional pair $(\cE,\cF)$ in $\cD$, the {\em left mutation}  $\cL_\cE \cF$ of $\cF$ by $\cE$ is defined by the triangle 
\[  \cL_\cE \cF[-1]\to \cE^0 \extto{f} \cF \to \cL_\cE \cF,  \]
where $\cL_\cE \cF$ is the cone of $f$, the minimal right $\add \{ \cE[i] \, | \, i\in \bbZ\}$-approximation of $\cF$.

 Dually, the {\em right mutation} $\cR_\cF \cE$ of $\cE$ by $\cF$ is defined by the triangle
\[  \cR_\cF \cE \to \cE \extto{g} \cF^0 \to \cR_\cF \cE[1]  \]
where $\cR_\cF \cE$ is the cocone of $g$, the minimal left $\add \{ \cF[i] \, | \, i\in \bbZ\}$-approximation of $\cE$.

For an exceptional pair $(E,F)$ in $\cH$ there are uniquely determined exceptional objects $L_E F$ and $R_F E$ in $\mathcal{H}$ which, up to translation in $\cD$, coincide with $\cL_\cE \cF$ and $\cR_\cF \cE$, respectively. This is due to the fact that the object $\cL_\cE \cF$ (resp.\ $\cR_\cF \cE$) is indecomposable.

The object $\cE^0$ in the approximation defining left mutation is concentrated in one shift. 
Thus the triangle defining left mutation corresponds to exactly one of the following short exact sequences
\begin{align*}
 (\bfE) & \quad \quad L_E F \into E^0 \stackrel{f}{\onto} F  \\
 (\bfM) & \quad \quad E^0 \stackrel{f}{\into} F \onto L_E F  \\
 (\bfX) & \quad f \colon F  \into L_E F \onto E^0 \phantom{\stackrel{f}{\into}} 
\end{align*}
depending on whether the approximation is an {\bf E}pimorphism, a {\bf M}onomorphism or an e{\bf X}tension, respectively. If the pair is such that $\Hom(E,F)=0$ and $\Ext^1(E,F)=0$, we have that $F \simeq L_E F$ and we call the mutation a {\em transposition}.

The analogous statements hold for right mutation of an exceptional pair $(E,F)$.

 For the rest of this subsection, unless stated otherwise, all
 exceptional sequences are considered to be in $\cH$. For an
 exceptional sequence  $\varepsilon=(E_1,\ldots,E_n)$, left mutation
 $\lambda_i$ and right mutation $\rho_i$ are defined by  
\begin{align*}
  \lambda_i \varepsilon :=& \, (E_1,\ldots,E_{i-1},L_{E_i}E_{i+1},E_i,E_{i+2},\ldots,E_n) \text{ and} \\
  \rho_i \varepsilon :=& \, (E_1,\ldots,E_{i-1},E_{i+1},R_{E_{i+1}}E_i,E_{i+2},\ldots,E_n).
\end{align*}
The mutations $\lambda_i$ and $\rho_i$ are mutually inverse (that is $\lambda_i \rho_i = 1 = \rho_i \lambda_i$), and satisfy the {\em braid relations}: $\lambda_i \lambda_{i+1} \lambda_i = \lambda_{i+1} \lambda_i \lambda_{i+1}$ for $1 \le i < n$ and $\lambda_i \lambda_j = \lambda_j \lambda_i$ for  $|i-j| \ge 2$. These relations define an action of the braid group $B_n$ in $n-1$ generators. This action is transitive on the set of complete exceptional sequences in $\cH$ (see Theorem~\ref{Braid.group.action.transitive}).

\subsubsection{Transitivity of the braid group action}

 In this subsection we consider the following two types of hereditary categories with a  tilting object: The module  category $\mod H$ of a finite dimensional hereditary algebra $H$ of infinite representation type, and the category $\cohX$ of coherent sheaves for some weighted projective line $\bbX$. Whenever we write $\cH$ we do not make a distinction between the two cases.

  The following are results we will use in this paper.

\begin{lemma}[{See \cite[Lemma 4.1]{HR}}] \label{lemma.HR} Let $E$ and $F$ be in $\cH$ and such that one of them is indecomposable. If $\Ext_\cH^1(E,F)=0$, then each  indecomposable morphism $f \in \Hom_\cH (E,F)$ is an epimorphism or a monomorphism.
\end{lemma}

\begin{lemma}[{\cite[Proof of Theorem 2]{Ringel3}}]\label{Hom.transposition} Let  $\varepsilon=(E_1,\ldots,E_n)$ be an exceptional sequence in $\cH$ and $a<b$ such that there exists $0 \ne f \in \Hom_\cH(E_a,E_b)$ but $\Hom_\cH(E_i,E_j)=0$ for the remaining pairs $(i,j)$ with $a \le i < j \le b$. Then we have {exactly one of} the following.
\begin{enumerate} 
\item The morphism $f$ is mono, and all right mutations in the sequence $\rho_{b-2} \cdots \rho_a$ are transpositions. 
\item The morphism $f$ is epi, and all right mutations in the sequence $\rho_{a+1} \cdots \rho_{b-1}$ are transpositions.
\end{enumerate}
\end{lemma}

\begin{proof} For the benefit of the reader, we reproduce the proof. By Lemma~\ref{lemma.HR}, we know that $f$ is mono or epi. 
\begin{enumerate} 
\item If $f$ is mono, we have induced epimorphisms $\Ext^1_\cH(E_b,E_i)\onto \linebreak[4]\Ext^1_\cH(E_a,E_i)$ for all $i$. The first $\Ext$ group vanishes for $i\le b$, and thus the second group vanishes for these $i$. Then we have that both $\Hom_\cH(E_a,E_i)=0=\Ext^1_\cH(E_a,E_i)$ for $a<i<b$, and thus $\rho_{b-2} \cdots \rho_a$ is a sequence of transpositions moving $E_a$ exactly behind $E_b$.
\item Dual to $(a)$. \qedhere
\end{enumerate}
\end{proof}

 \begin{lemma}{{ \cite[proof of Theorem~5]{Ringel3}}}\label{moving.closer.lemma} Let $\varepsilon=(E_{1},\ldots,E_{n})$ be an exceptional sequence in $\cH$ such that $\Ext^{1}(E_{a},E_{b})\ne 0$ with $b-a>1$ minimal. Then there exists a sequence of transpositions that reduce $b-a$ by at least one. 

 \end{lemma}

 \begin{proof}
{
 We recall the arguments. Let $t$ be maximal with $a \le t < b$ such that $\Hom_\cH(E_a,E_t)\ne 0$.

Assume $t > a$.  Let $0\ne f \in \Hom_\cH(E_a,E_t)$. By Lemma~\ref{lemma.HR} the morphism $f$ is either mono or epi. If $f$ is mono, then we have an induced epimorphism $\Ext^1_\cH(E_t,E_b) \onto \Ext^1_\cH(E_a,E_b)$. Since the latter space is nonzero so is the former one; a contradiction to the minimality of $b-a$. Thus $f$ is epi and $\Hom_\cH(E_t,E_i)=0$ for $t<i\le b$ (otherwise any nonzero map $g \from E_t \to E_i$ can be composed with $f$ to a nonzero map $E_a \to E_i$ contradicting the maximality of $t$). Now since $t>a$, we observe that $\Ext^1_\cH(E_t,E_i)=0$ for $t<i\le b$, and thus the sequence of transpositions $\rho_{b-1} \cdots \rho_{t}$ moves $E_t$ immediately after $E_b$. }
 
 Finally for $t=a$ we have that $\Ext^1_\cH (E_a,E_i)=0$ and $\Hom_\cH(E_a,E_i)=0$ for $a < i < b$, and thus the {sequence of} transpositions $\rho_{b-1} \cdots \rho_a$ move $E_a$ {immediately} before $E_b$.
 \end{proof}

\begin{lemma}[{\cite[Section~6]{Ringel3}}] \label{back.to.orthogonal} Let $\varepsilon=(E_1,\ldots,E_n)$ in $\mod H$  be an exceptional sequence, and $1\le i <n$. Then there exist $t \in \bbZ$ such that $\zeta:=(Z_1,\ldots,Z_n)=\rho_i^t\varepsilon$ satisfies $\Hom_H (Z_i,Z_{i+1})=0$.
\end{lemma}

We say that $\rho_i^t$ is a {\em proper reduction} for $\varepsilon$, provided that $\zeta=\rho_i^t\varepsilon$ satisfies $\Hom_H(Z_i,Z_{i+1})=0$ whereas $\Hom_H(E_i,E_{i+1})\ne 0$. An exceptional sequence is {\em orthogonal} if every exceptional pair {consisting of two (not necessarily consecutive) elements} of the sequence is orthogonal.

\begin{theorem}[{\cite[Theorem 2]{Ringel3}}] \label{Hom.arrow.reduction}
Any exceptional sequence  in $\mod H$ can be shifted by the braid group action to an  orthogonal sequence using only transpositions and proper reductions.
\end{theorem}

\begin{remark}
This theorem is not true for $\cohX$ since by  \cite[Proposition 2.8]{Meltzer} there are  no orthogonal exceptional sequences in $\cohX$.

The proof of the above theorem provides us with an algorithm to remove the $\Hom$-spaces between the objects of an exceptional sequence.
\end{remark}

\begin{theorem}[{\cite[Theorem 3]{Ringel3}}] The orthogonal complete  exceptional sequences in $\mod H$ are precisely those exceptional  sequences which consist of the simple modules.
\end{theorem}

\begin{theorem}[{\cite{CB,Meltzer,Ringel3}}]\label{Braid.group.action.transitive} The braid group acts transitively on the set of complete exceptional sequences in $\cH$.
\end{theorem}

\begin{proof} For $\mod H$ this was done in \cite{CB} and generalized to hereditary Artin algebras (not necessarily over an algebraically closed base field) in \cite[Section~7]{Ringel3}. For $\cohX$ this was done in \cite[Section~4]{Meltzer}.
\end{proof} 

An exceptional sequence $\varepsilon=(E_1,\ldots,E_n)$ in $\cH$ is said to be {\em strongly exceptional} provided that we have $\Ext^1_\cH(E_i,E_j)=0$ for all $i,j$. A strongly exceptional sequence that is complete will be called a {\em tilting sequence}.

\begin{theorem}[{\cite[Theorem 5]{Ringel3}}]\label{Ext.arrow.reduction} Any exceptional sequence in  $\mod H$ can be shifted by the braid group action to a strongly exceptional sequence using only transpositions and right mutations of type $(\bfX)$.
\end{theorem}

\begin{remark}
The proof of this statement provides us with an algorithm to reduce the $\Ext$-spaces between the objects of an exceptional sequence.
\end{remark}

As a special case of the above theorem we have the following lemma, which gives us a way to construct the indecomposable injective modules from the simple modules in $\mod H$ via mutation of exceptional sequences.

\begin{lemma} \label{from.simples.to.injectives.lemma}
Let $\cS=(S_1,\ldots,S_{n})$ and $\cI=(I_n,\ldots,I_{1})$ be the exceptional sequences of the simple and the injective $H$-modules for a hereditary algebra $H$, respectively. Moreover, let $\sigma_i= \lambda_i \cdots \lambda_{n-1}$ for $1\le i \le n-1$ and denote by $\sigma = \sigma_{n-1} \cdots \sigma_1$ the composition of these sequences of left mutations. Then $\sigma \cdot \cS = \cI$. Furthermore, the sequence $\sigma$ is just made up of left mutations of type $(\bfX)$.
\end{lemma}

\section{Sinks and sources in hereditary categories}
\label{section.sink.source}

In this section we introduce the notion of being a sink or a source for an indecomposable summand of a tilting object in $\cH$. This has been done for $\cH = \coh \mathbb X$ in \cite{Hubner1}, but here we extend this to any hereditary category $\cH$ with a tilting object.

\begin{definition}\label{definition.sink_source}
Let $T$ be a tilting object in $\cH$ and let $T_i$ be an 
indecomposable summand of $T$. 
\begin{enumerate}
\item We call $T_i$ a \emph{source} if there is a non-split
  monomorphism $T_i \to X$ for some $X\in\add T$.
\item We call $T_i$ a \emph{sink} if there is a non-split
  epimorphism $Y \to T_i$ for some $Y\in\add T$.
\end{enumerate}
\end{definition}

\begin{remark}
Later in this section we will give other characterizations of
sinks and sources. In particular, we will see that these properties are
mutually exclusive. Moreover, in the case when $H$ has no nonzero injectives, every indecomposable summand of a tilting object is either a source or a sink (see Proposition~\ref{sink.source}).
\end{remark}

The indecomposable summands of a tilting object which are either a sink or a source are precisely those in which one can mutate. That is, if an indecomposable summand of a tilting object is a sink (or a source), then we can always exchange it with another one to obtain a new tilting object. This is illustrated in the following proposition (see \cite{RiedtmannSchofield}).

\begin{proposition}
Let $T=T_i\oplus \bar T$ be a tilting object in $\cH$ where the summand
$T_i$ is indecomposable. Then we have the following.

\begin{enumerate}
\item The summand $T_i$ is a source if and only if the minimal left
  $\add(\bar T)$-approximation $f \from T_i \to  X$ is a
  monomorphism. In this case, for $T_i^* := \Coker(f)$, the object $T_i^* \oplus \bar
  T$ is a tilting object, and $T_i^*$ is a sink of this tilting object.
\item The summand $T_i$ is a sink if and only if the minimal right
  $\add(\bar T)$-approximation $g \from Y \to T_i$ is an
  epimorphism. In this case, for $T_i^* := \Ker(g)$, the object $T_i^* \oplus \bar T$ is a tilting object, and $T_i^*$ is a source of this tilting object. 
\end{enumerate}
\end{proposition}

\begin{proof}
If $f$ is a monomorphism then clearly $T_i$ is a source.
Assume that $T_i$ is a source. Let $h \from T_i \to Z$ be a non-split monomorphism with $Z \in \add(T)$. Since $h$ is non-split and $\End(T_i) = \mathbb{K}$, we may assume that there is no summand $T_i$ in $Z$. Thus $Z \in \add(\bar T)$. Now observe that $h$ factors through the $\add(\bar T)$-approximation $f$. Hence $f$ is also a monomorphism and we obtain a short exact sequence
\[ T_i \stackrel{f}{\into} X \stackrel{g}{\onto} T_i^* \]
where $T_i^*:=\Coker(f)$. In order to check that $\bar T \oplus T_i^*$
is a tilting object, we only need to verify $\Ext^1(\bar T,
T_i^*)=0=\Ext^1(T_i^*,\bar T)$, since the number of indecomposable
summands is right. This follows by applying $\Hom(\bar T, -)$ and
$\Hom(- ,\bar T)$ to the sequence above and using that $\Hom(f,\bar
T)$ is an epimorphism. This proves (a). Part (b) is dual.
\end{proof}

We now introduce some technical lemmas that will allow us to give
other characterizations of the sink/source property for an indecomposable summand of a tilting object in any hereditary category.

Let $\cH$ be a hereditary category with a tilting object $T$, and let $\Lambda = \End_{\cH}(T)$. We can then identify $D^b(\Lambda)$ with $D^b(\cH)$ via the {mutually inverse} equivalences
\begin{align*}
\R\Hom(T,-) & \colon D^b(\cH) \rightarrow D^b(\mod \Lambda) \mbox{ and}\\
T \otimes_{\Lambda}^{\L}- & \colon D^b(\mod \Lambda) \rightarrow D^b(\cH).
\end{align*}
Then we have the following:

\begin{lemma} \label{lemma.right_approx_mono_epi}
Let $T = \bar{T} \oplus T_i$ such that $T_i$ is indecomposable, and let $S_i = \Hom_{\cH}(T,T_i)/\rad(T,T_i)$ be the simple module corresponding to the projective module induced by $T_i$ in $\mod \Lambda$. Then the minimal right $\add \bar{T}$-approximation of $T_i$ is a monomorphism if and only if $T \otimes_{\Lambda}^{\L} S_i \in \cH$, and an epimorphism if and only if $T \otimes_{\Lambda}^{\L} S_i \in \cH[1]$.
\end{lemma}

\begin{proof}
Since $T$ is a tilting object, we know that $\mod \Lambda \subseteq \cH \vee \cH[1]$. Denoting by $P_t$ the indecomposable projective in $\mod \Lambda$ corresponding to an indecomposable summand $T_t$ of $T$, we get the projective resolution
\begin{align*}
0 \rightarrow  \oplus_{k} P_k \rightarrow \oplus_{j} P_j  \stackrel{\bar{f}}{\rightarrow}	P_i  \rightarrow	S_i \rightarrow 0
\end{align*}
 of $S_i$ (remember that $\pd S_i \leq 2$ since $\End_\cH(T)$ is quasitilted by \cite{HRS}). We denote the image of $\bar{f}$ by $\Omega S_i$. Then the above sequence can be decomposed into the two triangles
\[  \oplus_{k} P_k \rightarrow \oplus_{j} P_j  \stackrel{\bar{f}_1}{\rightarrow} \Omega S_i \rightarrow \oplus_k P_k[1] \text{ and } \Omega S_i \stackrel{\bar{f}_2}{\rightarrow} P_i \rightarrow S_i \rightarrow \Omega S_i[1] \]
in $D^b(\mod \Lambda)$, which correspond to the triangles
\begin{align*}
& \oplus_{k} T_k \rightarrow \oplus_{j} T_j \overset{f_1}{\rightarrow} T \otimes_{\Lambda}^{\L} \Omega S_i \rightarrow \oplus_k T_k[1] \qquad \text{and} \\
& T \otimes_{\Lambda}^{\L} \Omega S_i \extto{f_2} T_i \rightarrow T \otimes_{\Lambda}^{\L} S_i \to T \otimes_{\Lambda}^{\L} \Omega S_i[1]
\end{align*}
in $D^b(\cH)$. Here $f = f_2 \circ f_1$ is the minimal right $\add \bar{T}$-approximation of $T_i$. By Lemma~\ref{lemma.HR} we know that $f$ then must be either a monomorphism or an epimorphism.

If now $T \otimes_{\Lambda}^{\L} S_i$ is in $\cH$, we have the composition $[T_i \to T \otimes_{\Lambda}^{\L} S_i] \circ f = 0$ in $\cH$, and thus $f$ is not an epimorphism. Hence $f$ must be a monomorphism, and the lemma holds.

Next assume that $T \otimes_{\Lambda}^{\L} S_i$ is in $\cH[1]$. Similarly to the above paragraph we see that if $\oplus_k T_k \neq 0$ then $f$ is not a monomorphism, hence it is an epimorphism and the lemma holds. So assume instead that $\oplus_k T_k = 0$. Then $\oplus_j T_j \simeq T \otimes_{\Lambda}^{\L} \Omega S_i$, and thus $\Cone f = \Cone f_2 = T \otimes_{\Lambda}^{\L} S_i \in \cH[1]$ meaning that $f$ is an epimorphism.
\end{proof}

This lemma has a natural dual, which is obtained by identifying along the equivalence $\D\R\Hom(-, T) \colon D^b(\cH) \to D^b(\mod \Lambda)$. Observe that if we denote by $\nu$ the Serre functor of the derived categories, then $\D\R\Hom(-, T) = \R\Hom(\nu^{-1} T, -)$, and thus the inverse equivalence is given by $\nu^{-1}T \otimes_{\Lambda}^{\L} -$.

\begin{lemma}
Let $T = \bar{T} \oplus T_i$ such that $T_i$ is indecomposable, and let $S_i = \D(\Hom_{\cH}(T_i,T)/\rad(T_i,T))$ be the simple module corresponding to the injective module induced by $T_i$ in $\mod \Lambda$. Then the minimal left $\add \bar{T}$-approximation of $T_i$ is an epimorphism if and only if $\nu^- T \otimes_{\Lambda}^{\L} S_i \in \cH$, and a monomorphism if and only if $\nu^- T \otimes_{\Lambda}^{\L} S_i \in \cH[-1]$.
\end{lemma}

Note that the derived equivalences between $\cH$ and $\mod \Lambda$ commute with the Serre functor $\nu$, and hence $(\nu T) \otimes_{\Lambda}^{\L} S_i = \nu(T \otimes_{\Lambda}^{\L} S_i)$. Thus we can apply $\nu$ in the above lemma, and obtain the following.

\begin{lemma} \label{lemma.left_approx_mono_epi}
Let $T = \bar{T} \oplus T_i$ such that $T_i$ is indecomposable, and
let $S_i = \Hom_{\cH}(T,T_i)/\rad(T,T_i)$ be the simple module
corresponding to the projective module induced by $T_i$ in $\mod
\Lambda$. Then the minimal left $\add \bar{T}$-approximation of $T_i$
is an epimorphism if and only if $T \otimes_{\Lambda}^{\L} S_i \in
\nu\cH$, and a monomorphism if and only if $T \otimes_{\Lambda}^{\L} S_i \in \nu\cH[-1]$.
\end{lemma}

Putting together Lemmas~\ref{lemma.right_approx_mono_epi} and \ref{lemma.left_approx_mono_epi} we get the following proposition:

\begin{proposition}\label{sink.source}
Let $T=T_i \oplus \bar{T}$ be a tilting object in $\cH$, where $T_i$ is indecomposable. We set $\Lambda = \End_{\cH}(T)$, and denote by $S_i$ the simple $\Lambda$-module corresponding to $T_i$. Then we have exactly one of the following (where $\inj(\cH)$ denotes the injective objects in $\cH$):
\begin{enumerate}
\item $T \otimes_{\Lambda}^{\L} S_i \in \cH \backslash \inj(\cH)$. In this case both the left and right minimal $\add \bar T$-approximations are monomorphisms, that is, $T_i$ is a source and not a sink.
\item $T \otimes_{\Lambda}^{\L} S_i \in (\cH \backslash \inj(\cH))[1]$. In this case both the left and the right minimal $\add \bar T$-approximations are epimorphisms, that is, $T_i$ is a sink and not a source.
\item $T \otimes_{\Lambda}^{\L} S_i \in \inj(\cH)$. In this case the minimal right $\add \bar T$-approximation is a monomorphism and the minimal left $\add \bar T$-approximation is an epimorphism, that is, $T_i$ is neither a sink nor a source.
\end{enumerate}
\end{proposition}

As an easy consequence we obtain the following result of Happel and Unger (\cite[Proposition~3.6]{HU}).

\begin{corollary} \label{corollary.always_2_compl}
Let $\cH$ be a hereditary category with a tilting object and with no nonzero projectives. Then any indecomposable summand of a tilting object in $\cH$ is either a sink or a source. In particular every almost complete tilting object has two complements.
\end{corollary}

If $\cH$ has nonzero injectives (that is, if $\cH$ is $\mod H$ for some hereditary algebra $H$), there exist tilting objects with indecomposable summands which are neither sinks nor sources. It is not possible to mutate in such summands, since the corresponding almost complete tilting objects have only one complement. In the following, we present a way to circumnavigate this problem.

Let $H$ be a representation infinite connected hereditary algebra.

\begin{definition}
Let $T \in \mod H$ be a tilting object such that no summands of $T$ are in the preinjective component of the AR-quiver of $\mod H$. Then we define an indecomposable summand $T_i$ of $T$ to be a \emph{source$^{-}$} if $\tau^{-n} T_i$ is a source in $\tau^{-n}T$ for almost all $n>0$, and we define $T_i$ to be a \emph{sink$^{-}$} if $\tau^{-n}T_i$ is a sink in $\tau^{-n}T$ for almost all $n>0$.
\end{definition}

\begin{remark}
If $T$ has preinjective direct summands we can always replace it by a different tilting module without preinjective direct summands which gives rise to the same cluster tilted algebra. Indeed, suppose $T = T_I \oplus \bar{T}$, with $T_I$ preinjective and $\bar{T}$ without preinjective direct summands. Clearly there is $\ell$ such that $\tau^{- \ell} T_I = 0$, or, equivalently, such that $\tau_{D^b(H)}^{- \ell} T_I \in (\mod H)[1]$. Now
\[ \tau_{D^b(H)}^{-\ell-1} \bar{T} \oplus F^{-1} \tau_{D^b(H)}^{-\ell-1} T_I = \tau_{D^b(H)}^{-\ell-1} \bar{T} \oplus \tau_{D^b(H)}^{-\ell} T_I[-1] \in \mod H \]
is a tilting module without preinjective direct summands giving rise to the same cluster tilted algebra.
\end{remark}

\begin{remarks} ~

\begin{enumerate}
\item If $\tau^{-k}T_i$ is a sink in $\tau^{-k}T$ for some $k\geq 0$, then $\tau^{-k'}T_i$ will be a sink in $\tau^{-k'}T$ for any $k'>k$ (see Proposition~\ref{sink.source}). 
\item There is a connection between being a source or a sink in the quiver of $\End_{\cH}(T)$ and being a source or a sink in $T$ as defined above: a source (resp.\ sink) in $\End_{\cH}(T)$ can never be a sink (resp.\ source) in $T$ (though it can be neither), and thus will be a $\nsource$ (resp.\ $\nsink$) in $T$.
\end{enumerate}
\end{remarks}

We can make a hereditary category $\cH$ without nonzero projectives from $\mod H$ by letting $\cH = \cI[-1] \vee \cP \vee \cR$ where $\cI$ is the preinjective component, $\cP$ is the preprojective component and $\cR$ is the union of the regular components of the AR-quiver of $\mod H$. Then we have the following lemma:

\begin{lemma}
Let $H$ be a representation infinite indecomposable hereditary algebra, and let $T \in \mod H$ be a tilting object such that no summands of $T$ are preinjective. If $T_i$ is an indecomposable summand of $T$ we have the following:
\begin{enumerate}
\item $T_i$ is a source$^{-}$ if and only if $T_i$ is a source in $\cH$ and if and only if $T \otimes_{\Lambda}^{\L} S_i \in \cP \vee \cR$.
\item $T_i$ is a sink$^{-}$ if and only if $T_i$ is a sink in $\cH$ and if and only if $T \otimes_{\Lambda}^{\L} S_i \in \cI \vee \cP[1] \vee \cR[1]$.
\end{enumerate}
\end{lemma}

\begin{proof}
By definition of source$^{-}$ we know that $T_i$ is a source$^{-}$ in $T$ if and only if $\tau^{-n} T_i$ is a source in $\tau^{-n}T$ for $n \gg 0$. This is the case if and only if $(\tau^{-n}T) \otimes_{\Lambda}^{\L} S_i$ is in $\mod H$ for $n \gg 0$ by Proposition~\ref{sink.source}, and since we have that $(\tau^{-n}T) \otimes_{\Lambda}^{\L} S_i = \tau^{-n}(T \otimes_{\Lambda}^{\L} S_i)$ we see that this is the case if and only if $T \otimes_{\Lambda}^{\L} S_i$ is in $\cP \vee \cR$. This shows (i). The argument for (ii) is similar.
\end{proof}

In particular we see that for a tilting module in $\mod H$ such that $T \in \cP \vee \cR$, any indecomposable summand $T_i$ must either be a \nsource or a $\nsink.$ In this sense we can always do tilting mutation by replacing T with $\tau^{-n}T$ for a sufficiently large positive integer $n$.

\begin{example}
Let $H = \bbK[1 \, \tikz[baseline]{ \draw [<-] (0,.15) -- (.5,.15); \draw [<-] (0,.05) -- (.5,.05);} \, 2 \, \tikz[baseline]{ \draw [<-] (0,.1) -- (.5,.1);} \, 3]$. The beginning of the preprojective component looks as follows (numbers are dimension vectors):
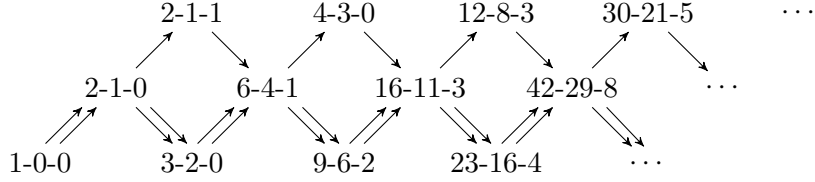
\[ \begin{tikzpicture}
 \node (P11) at (0,0) {1-0-0};
 \node (P21) at (1,1) {2-1-0};
 \node (P31) at (2,2) {2-1-1};
 \node (P12) at (2,0) {3-2-0};
 \node (P22) at (3,1) {6-4-1};
 \node (P32) at (4,2) {4-3-0};
 \node (P13) at (4,0) {9-6-2};
 \node (P23) at (5,1) {16-11-3};
 \node (P33) at (6,2) {12-8-3};
 \node (P14) at (6,0) {23-16-4};
 \node (P24) at (7,1) {42-29-8};
 \node (P34) at (8,2) {30-21-5};
 \node (P15) at (8,0) {$\cdots$};
 \node (P25) at (9,1) {$\cdots$};
 \node (P35) at (10,2) {$\cdots$};
 \draw [->] (P11.75) to (P21.210);
 \draw [->] (P11.45) to (P21.225);
 \draw [->] (P21) -- (P31);
 \draw [->] (P21.315) to (P12.135);
 \draw [->] (P21.330) to (P12.105);
 \draw [->] (P31) -- (P22);
 \draw [->] (P12.75) to (P22.210);
 \draw [->] (P12.45) to (P22.225);
 \draw [->] (P22) -- (P32);
 \draw [->] (P22.315) to (P13.135);
 \draw [->] (P22.330) to (P13.105);
 \draw [->] (P32) -- (P23);
 \draw [->] (P13.75) to (P23.210);
 \draw [->] (P13.45) to (P23.225);
 \draw [->] (P23) -- (P33);
 \draw [->] (P23.315) to (P14.135);
 \draw [->] (P23.330) to (P14.105);
 \draw [->] (P33) -- (P24);
 \draw [->] (P14.75) to (P24.210);
 \draw [->] (P14.45) to (P24.225);
 \draw [->] (P24) -- (P34);
 \draw [->] (P24.315) to (P15.135);
 \draw [->] (P24.330) to (P15.105);
 \draw [->] (P34) -- (P25);
\end{tikzpicture} \]
Using the dimension vectors it is easily checked which approximations are monomorphisms and epimorphisms, respectively.

For $T = H = P_1 \oplus P_2 \oplus P_3$ we see that $P_1$ and $P_2$ are sources, while $P_3$ is not exchangeable.

For $\tau^- T = \tau^- P_1 \oplus \tau^- P_2 \oplus \tau^- P_3$ one checks that $\tau^- P_1$ is a source, $\tau^- P_2$ is not exchangeable, and $\tau^- P_3$ is a sink.

For $\tau^{-2} T = \tau^{-2} P_1 \oplus \tau^{-2} P_2 \oplus \tau^{-2} P_3$ we have that $\tau^{-2} P_1$ is a source, while $\tau^{-2} P_2$ and $\tau^{-2} P_3$ are sinks.

Thus in $T$ we have that $P_1$ is a source and a $\nsource$, $P_2$ is a source and a $\nsink$, and $P_3$ is not exchangeable and a $\nsink$.
\end{example}

The distribution of sinks and sources (or $\nsink$s and $\nsource$s if
$\cH$ has nonzero injectives) will play a key role in defining a graded mutation rule for cluster-tilting objects in the next section.

 We have seen that one way to determine the sink-source distribution of a tilting object $T \in \cH$ is by calculating the position in $\cD:=D^b(\cH)$ of $S_1,\ldots,S_n$ the simple $\Lambda$-modules, where $\Lambda=\End_\cH(T)$. Another way is by using a result from \cite{Hubner2} that exploits free left and right mutations of exceptional sequences in $\cD$, which we now define.

\begin{definition}[{\cite{Hubner2}}] \label{def.free.mutation} For an exceptional pair $(\cE,\cF)$ in $\cD$ the {\em free} left and right mutations are defined as
\[
  \tcL_{\cE} \cF := \cL_{\cE}\cF[1] \text{ and }   \tcR_{\cF} \cE := \cR_{\cF}\cE[-1].
\]
 Similarly, for an exceptional sequence $\varepsilon=(\cE_1,\ldots,\cE_n)$ one defines $\lambdat_i$  (resp.\ $\rhot_i$) acting on exceptional sequences by composing with the functor $[1]$ (resp.\ $[-1]$) in the corresponding position.
\end{definition}

The following proposition explains the behavior of the sink-source distribution after (tilting) mutation by using the simple modules of quasitilted algebras.

\begin{proposition}[{\cite[Proposition~3.4]{Hubner2}}] \label{prop.behavior.sink-source} Let $T = \oplus_{i=1}^{n}T_{i}$ be a tilting object in a hereditary category $\cH$ without nonzero projectives. Denote by $\Lambda $ the quasitilted algebra $\End_{\cH}(T)$ and let $S_{1},\ldots,S_{n}$ denote the simple $\Lambda$-modules. Fix an indecomposable summand $T_{j}$ of $T$ for some $1\le j \le n$. Let $T^{(j)}$ be the tilting object obtained by replacing the indecomposable summand $T_{j}$ with $T_{j}^*$. Denote by $\Lambda^{(j)}=\End_{\cH}(T^{(j)})$ and $S_{1}^{(j)},\ldots,S_{n}^{(j)}$ the simple $\Lambda^{(j)}$-modules. Then we have the following.

\begin{enumerate} 
\item If $T_{j}$ is a sink of $T$, we have the following.
 \begin{enumerate}
  \item The simple $S_{j}^{(j)}$ is isomorphic to $S_{j}[-1]$.
  \item For every $k$ such that there is an irreducible morphism $T_{k}\to T_{j}$ in $\add_{\cH}(T)$ we have that
    \[   S_{k}^{(j)} \simeq \tcL_{S_{j}} S_{k}, \]
 that is, the simple $S_{k}^{(j)} $ is obtained by free left mutation of $S_{k}$ over $S_{j}$.
   \item All remaining simple $\Lambda$ and $\Lambda^{(j)}$-modules coincide.
 \end{enumerate}

\item If $T_{j}$ is a source of $T$, we have the following.
 \begin{enumerate}
  \item The simple $S_{j}^{(j)}$ is isomorphic to $S_{j}[1]$.
  \item For every $k$ such that there is an irreducible morphism $T_{j}\to T_{k}$ in $\add_{\cH}(T)$ we have that
    \[   S_{k}^{(j)} \simeq \tcR_{S_{j}} S_{k}, \]
 that is, the simple $S_{k}^{(j)} $ is obtained by free right mutation of $S_{k}$ over $S_{j}$.
   \item All remaining simple $\Lambda$ and $\Lambda^{(j)}$-modules coincide.
 \end{enumerate}
  
\end{enumerate}
\end{proposition}

Together with Proposition \ref{sink.source} this shows that only the object we mutate in and its direct neighbors can change from being a sink to being a source or vice versa. In fact, we can be even more specific:

\begin{proposition} Using the notation of Proposition~\ref{prop.behavior.sink-source}, we have the following.
\begin{enumerate}
 \item If $T_{j}$ is a sink of~$T$ and there is an irreducible morphism $T_{k}\to T_{j}$ in $\add_{\cH}(T)$ such that $T_{k}$ is a sink of $T$, then $T_{k}$ is a sink of $T^{(j)}$.

 \item If $T_{j}$ is a source of $T$ and there is an irreducible morphism $T_{j}\to T_{k}$ in $\add_{\cH}(T)$ such that $T_{k}$ is a source of $T$, then $T_{k}$ is a source of $T^{(j)}$.
\end{enumerate}
\end{proposition}

\begin{proof} \mbox{}
\begin{enumerate}
 \item If $T_{k}$ is a sink then $f\from X \to T_{k}$, the minimal right $\add_{\cH}(T/T_{k})$-approximation of $T_{k}$ is an epimorphism.  Let $g\from Y \to T_{k}$ be the minimal right $\add_{\cH}(T^{(j)}/T_{k})$-approximation of  $T_{k}$. It is not hard to see that $f$ factors through $g$, and therefore $g$ is also an epimorphism. Hence the result follows.
 
 \item Dual to (a). \qedhere
\end{enumerate}
\end{proof}

\section{Graded mutation} \label{section.gradmut}

In this section we develop the theory of graded mutation for cluster-tilted algebras. This is an adaptation of the mutation rule for tilting sheaves given in \cite{Hubner1} to the cluster category. The aim is to be able to mutate the quiver of a cluster-tilted algebra without losing information about the quiver of the underlying quasitilted algebra. In order to do this, we have to choose a ``canonical'' hereditary category without nonzero projectives, from which we construct the cluster category and its grading.

 Throughout this section, define $\cH_{*}$ to be either the hereditary category $\cohX$ for some weighted projective line $\bbX$, or $\cI[-1]\vee \cP \vee \cR$ where $\cP$ (resp.\ $\cR,\cI$) is the preprojective (resp.\ regular, preinjective) component of a hereditary algebra $H$ of infinite representation type. Observe that both types of categories coincide when $\mod H$ is tame and $\cohX$ is tame domestic. In this setup the Auslander-Reiten translate $\tau\from \cH_{*} \to \cH_{*} $ is an equivalence, and $\cH_{*}$ has no nonzero projectives (resp.\ injectives). Note that every almost complete tilting object in $\cH_*$ has exactly two complements (see \cite[Corollary~0.4]{Hubner3} for the $\cohX$ case, and Corollary~\ref{corollary.always_2_compl} for the general case).
  
Let $\cC_{*}:=\cC_{\cH_{*}}$ be the cluster category of $\cH_{*}$. Then $\cC_{*}$ has the same objects as $\cH_{*}$ and $\bbZ$-graded morphism spaces by
\[ \Hom_{\cC_{*}}(X,Y):= \underbrace{\Hom_{\cH_{*}}(X,Y)}_{\text{degree } 0} \oplus \underbrace{\Ext_{\cH_{*}}^1(X,\tau^{-1}Y)}_{\text{degree } 1} \text{ for } X,Y\in\cC_{*}. \]
We call this grading the {\em natural} grading of $\cC_*$. There is a natural one-to-one correspondence between tilting objects in $\cH_*$ and cluster-tilting objects in $\cC_*$ (\cite[Proposition~3.4]{BMRRT}). For a tilting object $T \in \cH_*$  we also write $T$ for its image in $\cC_*$.

Fix a cluster-tilting object $T\in\cC_{*}$  and denote its endomorphism ring by $\Gamma=\End_{\cC_{*}}(T)$. Then $\Gamma$ can be
seen as the trivial extension $\End_{\cH_{*}}(T) \ltimes\Ext_{\cH_{*}}^1(T,\tau^- T)$ (\cite[proof of 3.1]{Z}). It was observed in \cite{ABS} that the bimodules $\Ext^1_{\cH_{*}}(T,\tau^- T)$ and $\Ext^2_{\Lambda }(\D\Lambda ,\Lambda )$ are isomorphic, where $\Lambda =\End_{\cH_*}(T)$ and $\D$ is the usual $\bbK$-duality (see \cite[Section~3.1]{Ringel4} for an elementary proof without using derived categories). Hence, one can also define $\Gamma$  as the trivial extension $\Lambda  \ltimes \Ext^2_{\Lambda
}(\D\Lambda ,\Lambda )$. The isomorphism between $\Ext^1_{\cH_{*}}(T,\tau^- T)$ and $\Ext^2_{\Lambda }(\D\Lambda ,\Lambda )$ also gives us a one-to-one correspondence between the irreducible morphisms of degree one in $\add_{\cC_*} T$ and the minimal relations of $\Lambda $, and we identify these two sets\footnote{Note that the minimal relations go in the opposite direction of the degree-one arrows.}. 
  
For $\cH_*=\cohX$, the identification above allows us to talk about rank as an additive function for cluster-tilting objects in $\cC_{*}$ by letting $\rank_{\cC_*} X := \rank_{\cH_*}X$ for $X\in \cC_*$. Then the following is an immediate consequence of Theorem~\ref{rank-is-additive}.

\begin{theorem} Let $\cH_*=\cohX$ for some weighted projective line $\bbX$ and let $T=\oplus_{i=1}^n T_i$ be a cluster-tilting object in $\cC_*$, with $T_i$ indecomposable. Then $\rank_{\cC_*}$ is an additive function on $T$, i.e.

\begin{itemize}
\item[(i)]  $\rank_{\cC_*}(T_i) \geq 0$ for all $1 \leq i \leq n$, and
\item[(ii)] $2 \cdot \rank_{\cC_*}(T_i) = \sum\limits_{j} a(j,i) \cdot \rank_{\cC_*}(T_j)  
- \sum\limits_{j} b(j,i) \cdot \rank_{\cC_*}(T_j)$.
\end{itemize}
Here $a(i,j)$ denotes the number of arrows between the vertices in $Q_T$ corresponding to degree zero morphisms between $T_i$ and $T_j$ (in either direction, $a(i,j)=0$ if there are no such arrows), and similarly $b(i,j)$ denotes the number of arrows between the vertices corresponding to degree one morphisms between $T_i$ and $T_j$.
\end{theorem}

  Since we have a $\bbZ$-grading on the morphism spaces, the quiver $Q_{\Gamma}$ of $\Gamma=\End_{\cC_{*}}(T)$ inherits a natural $\bbZ$-grading. One natural question to ask is if one can adapt the mutation rule for quivers of cluster-tilted algebras to carry the grading information. 

 This is in fact possible, but before we introduce this graded mutation rule, we present some propositions that give  a better understanding of the the $\bbZ$-grading for the endomorphism ring of cluster-tilting objects.

We call an indecomposable summand of a cluster-tilting object a {\em sink} (resp.\ {\em source}), if it is a sink (resp.\ source) of the corresponding tilting object in $\cH_*$ (see Section~\ref{section.sink.source}). Then we have the following.

\begin{proposition}
Let $T$ be a cluster-tilting object in $\cC_*$. Let $f \from T_i \to T_j$ be a homogeneous irreducible morphism in $\cC_*$ between indecomposable objects $T_i, T_j \in \add T$. Then the degree of $f$ is one if and only if $T_i$ is a sink and $T_j$ is a source. In particular any two homogeneous irreducible morphisms from $T_i$ to $T_j$ have the same degree.
\end{proposition}

\begin{proof}
Let $S_i,S_j$ be the two simple  $\Lambda$-modules corresponding to the tops of $P_i:=\Hom_{\cH_*}(T,T_i)$ and $P_j:=\Hom_{\cH_*}(T,T_j)$, respectively. Then we have
\begin{align*}
\deg f = 0 & \Longrightarrow \Ext_{\Lambda}^1(S_j, S_i) \neq 0 \\
\deg f = 1 & \Longrightarrow \Ext_{\Lambda}^2(S_i, S_j) \neq 0.
\end{align*}
Going back to the derived category of $\cH_*$, we obtain
\begin{align*}
\deg f = 0 & \Longrightarrow \Hom_{D^b(\cH_*)}(T \otimes_{\Lambda}^{\L} S_j, T \otimes_{\Lambda}^{\L} S_i[1]) \neq 0 \\
\deg f = 1 & \Longrightarrow \Hom_{D^b(\cH_*)}(T \otimes_{\Lambda}^{\L} S_i, T \otimes_{\Lambda}^{\L} S_j[2]) \neq 0.
\end{align*}
Since $T \otimes_{\Lambda}^{\L} S_i$ and $T \otimes_{\Lambda}^{\L} S_j$ lie in $\cH_*$ or $\cH_*[1]$, and since $\cH_*$ is hereditary, it follows that
\begin{align*}
\deg f = 0 & \Longrightarrow \text{not }( T \otimes_{\Lambda}^{\L} S_j \in \cH_* \text{ and } T \otimes_{\Lambda}^{\L} S_i \in \cH_*[1]) \\
\deg f = 1 & \Longrightarrow T \otimes_{\Lambda}^{\L} S_i \in \cH_*[1] \text{ and } T \otimes_{\Lambda}^{\L} S_j \in \cH_*.
\end{align*}
Now the claim follows from Proposition~\ref{sink.source}.
\end{proof}

As an immediate consequence we obtain the following.

\begin{corollary} \label{corollary.ss_gives_grading} 
Let $Q$ be the quiver of a cluster-tilting object in $\cC_*$. Assume that for each vertex of $Q$ we know whether it corresponds to a sink or to a source. Then we can recover the $\bbZ$-grading of the arrows of $Q$.
\end{corollary}

 Unfortunately, the converse is not true, as the following example illustrates.

\begin{example}\label{example.several_sink_source_dist}
Let $\cH_*=\cohX$ where $\bbX$ is a weighted projective line of type $(2,2,2,2)$. The following sink-source distributions are valid for the canonical algebra in $\cH_*$. 
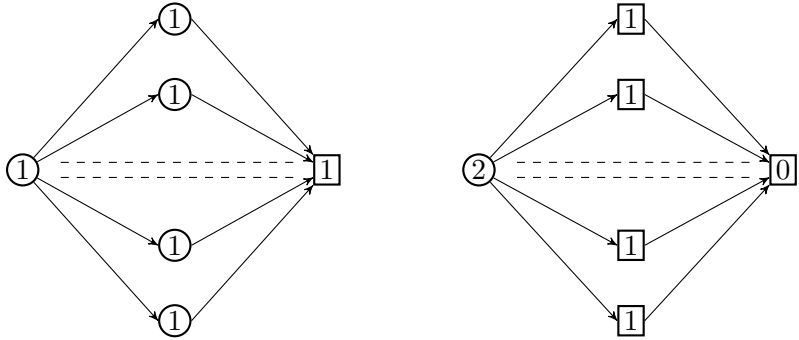
\[ 
\begin{tikzpicture}[xscale=2,yscale=-2]
 \node (1-0) at (0,1) [source] {$1$};
 \node (1-1) at (1,0) [source] {$1$};
 \node (1-2) at (1,0.5) [source] {$1$};
 \node (1-3) at (1,1.5) [source] {$1$};
 \node (1-4) at (2,1) [sink] {$1$};
 \node (1-5) at (1,2) [source] {$1$};
 \node (2-0) at (3,1) [source] {$2$};
 \node (2-1) at (4,0) [sink] {$1$};
 \node (2-2) at (4,0.5) [sink] {$1$};
 \node (2-3) at (4,1.5) [sink] {$1$};
 \node (2-4) at (5,1) [sink] {$0$};
 \node (2-5) at (4,2) [sink] {$1$};
\foreach \x in {1,2}{
 \draw [->] (\x-0) -- (\x-1.west);
 \draw [->] (\x-0) -- (\x-2.west);
 \draw [->] (\x-0) -- (\x-3.west);
 \draw [->] (\x-0) -- (\x-5.west);
 \draw [->] (\x-1.east) -- (\x-4);
 \draw [->] (\x-2.east) -- (\x-4);
 \draw [->] (\x-3.east) -- (\x-4);
 \draw [->] (\x-5.east) -- (\x-4);
}
\draw [dashed] (0.25,0.95) --  (1.8,0.95);
\draw [dashed] (0.25,1.05) -- (1.8,1.05);
\draw [dashed] (3.25,0.95) --  (4.8,0.95);
\draw [dashed] (3.25,1.05) -- (4.8,1.05);

\end{tikzpicture}
\]
Here the numbers denote the ranks of the indecomposable objects, circles symbolize sources, and squares symbolize sinks. Observe that different rank distributions give rise to different sink-source distributions.
\end{example}

 We are now ready to present the graded mutation rule. This rule was explicitly calculated in \cite[Corollary 4.16]{Hubner1} for tilting sheaves in $\cohX$, although the same proof works for $\cH_*$. If we forget the grading, it coincides with the quiver mutation rule of Fomin-Zelevinsky.  Therefore, we only explain how to mutate the grading. In the following, the degree zero morphisms are depicted with solid arrows and the degree-one morphisms with dashed arrows. This is done to emphasize the correspondence between minimal relations of the tilting object and degree one morphisms of the corresponding cluster-tilting object.

\begin{proposition}[Graded mutation rule] Let $T=T_l \oplus \bar{T}$ be a cluster-tilting object in $\cC_{*}$ with $T_l$ indecomposable. Denote by $T^*_l$ the other complement of $\bar{T}$. Then $T_l$ is a sink if and only if  $T^*_l$ is a source. Furthermore, if $T_l$ is a sink, we have the following.
\begin{enumerate}
\item Any degree zero arrow $T_i \stackrel{a}{\to} T_l$ is turned into a degree zero arrow $T_l^* \stackrel{a^*}{\to} T_i$.
\item Any degree zero arrow $T_l \stackrel{a}{\to} T_i$ is turned into a degree one arrow $T_i \stackrel{a^*}{\to} T_l^*$.
\item Any degree one arrow $T_l \stackrel{a}{\to} T_i$ is turned into a degree zero arrow $T_i \stackrel{a^*}{\to} T_l^*$.
\item The degree of new arrows which are formal compositions $[fg]$ is the sum of the degrees of $f$ and $g$.
\end{enumerate} 
\end{proposition}

\begin{remark}
There is a dual statement for the case when $T_l$ is a source.
\end{remark}

 Observe that in order to apply the graded mutation rule, we must know whether the indecomposable summand of the cluster-tilting object we are going to replace is a sink or a source. Note that we may lose information about the sink-source distribution after graded mutation (see Example~\ref{example.lose.sink-source.rank}).
 
As a first application of the graded mutation rule we present the following proposition, which clarifies to what extent we can have a converse for Corollary~\ref{corollary.ss_gives_grading}. 

\begin{proposition}\label{prop.uniqueness_sink-source} 
Let $\cH_*$ be as in the beginning of the section. We write $\cH_{H}$ if $\cH_*$ is constructed from $\mod H$ and $\cH_{\bbX}$ if $\cH_*=\cohX$. Similarly, we write $\cC_H$ (resp.\ $\cC_{\bbX}$) if the cluster category is constructed from $\cH_H$ (resp.\ $\cH_{\bbX}$). Then we have the following.

\begin{enumerate}
\item Let $Q$ be the (ungraded) quiver of a cluster-tilting object in $\mathcal{C}_H$. Then $Q$ has a unique sink-source distribution and a unique grading.
\item Let $Q$ be the graded quiver of a cluster-tilting object in $\mathcal{C}_{\bbX}$. If $\bbX$ is of tubular type, then assume additionally that the cluster tilting object has at least one indecomposable summand of rank $0$ (by using tubular mutations any cluster-tilting object can be turned into one satisfying this assumption). Then $Q$ has a unique sink-source distribution. Moreover, one can use the grading to recover this sink-source distribution and, by graded mutation, one can transfer this information to the mutation component.
\end{enumerate}

\end{proposition}
\begin{proof}

\mbox{}
\begin{enumerate}
\item
In this case, there exists a sequence of mutations
\[ T = T^{(0)} \underset{\mu_{i_1}}{\sim} T^{(1)} \underset{\mu_{i_2}}{\sim} \cdots \underset{\mu_{i_n}}{\sim} T^{(n)} \]
starting in the cluster tilting object $T$ corresponding to $Q$, and with $\End_{\cC_H}(T^{(n)}) = H$. Then we may assume $T^{(n)} = H$, and we can obtain the other $T^{(j)}$ explicitly as  $T^{(j)} = \mu_{i_{j+1}} \cdots \mu_{i_n} H$.

Note that, if we want to mutate in $\mod H$ (rather than in $\cH_H$), we may need to replace the tilting modules $T^{(j)}$ by a sufficiently high shift $\tau^{-n_j} T^{(j)}$ in order to obtain a well-defined sink-source distribution. Observe that we are able to calculate the sink-source distribution at each step, since we know the modules explicitly. Finally, by Corollary~\ref{corollary.ss_gives_grading} we can recover the grading of $Q$ uniquely by using the sink-source distribution of $T$.

\item  For the $\cH_{\bbX}$ case, we can use the $\bbZ$-grading to calculate the ranks of the indecomposable summands of the cluster-tilting object $T$ corresponding to $Q$ via the Cartan matrix. If $\bbX$ is not tubular, the radical of the quadratic form has rank one, and  the ranks of the indecomposable summands of $T$ are determined in a unique way by \cite[Lemma 2.5]{LM1}. If $\bbX$ is of tubular type, the radical of the quadratic form has rank two. However, the ranks of the indecomposable summands of $T$ are still uniquely determined by our additional assumption that there is a summand of rank $0$. Seeing that morphisms of degree one go from sinks to sources, it remains only to calculate the sink-source distribution at the indecomposable summands of $T$ where all adjacent morphisms have degree zero. Let $T_{0}$ be such a summand and assume $\rank_{\cC_{\bbX}} T_{0} \neq 0$. Then $T_{0}$ is a source (resp.\ a sink) if and only if
\[ \rank_{\cC_{\bbX}} T_0 \leq \sum_{\substack{\text{arrows} \\ T_0 \to T_i}} \rank_{\cC_{\bbX}} T_i \qquad (\text{resp.\ } \sum_{\substack{\text{arrows} \\ T_i \to T_0}} \rank_{\cC_{\bbX}} T_i > \rank_{\cC_{\bbX}} T_0 ). \]
Finally note that the subquiver formed by the summands of rank $0$ is a disjoint union of (graded) quivers of cluster-tilted algebras of type $A$, with connecting vertices (see \cite{Dagfinn}) where they are connected to the summands of positive rank. For such quivers the distribution of sinks and sources can be determined by comparing to the cluster-tilted algebras of type $A$.

Observe that once we know the sink-source distribution, the graded mutation rule gives us the grading after mutation. Hence we do not lose information when mutating. \qedhere
\end{enumerate}
\end{proof}

We now illustrate Proposition~\ref{prop.uniqueness_sink-source}~{(b)} with an example.

\begin{example} \label{example.lose.sink-source.rank} Let $\cC_{\bbX}$ be the cluster category of a weighted projective line $\bbX$ of type $(3,3,4)$. Denote the canonical cluster-tilting object by $T=\oplus_{0\le \vx \le \vc} \cO(\vx)$. Let $Q_{1}$ be the quiver of the endomorphism ring obtained by mutating $T$ at $\cO$. At each vertex, the first number indicates the number of the vertex and the second number (in bold face) denotes the rank of the corresponding indecomposable summand of the cluster-tilting object. The circles denote the sources and the rectangles the sinks.
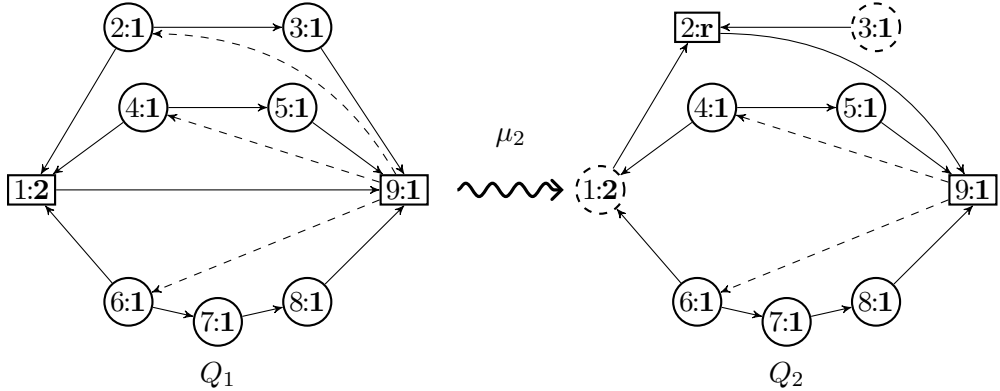
\[\scalebox{1}{
\begin{tikzpicture}[scale=0.7]
 \node (1) at (-3.5,0) [sink] {$1\sep {\mathbf 2}$};
 \node (2) at (-1.68,3.08) [source] {$2\sep {\mathbf 1}$};
 \node (3) at (1.68,3.08) [source]  {$3\sep {\mathbf 1}$};
 \node (4) at (-1.4,1.56) [source]  {$4 \sep {\mathbf 1}$};
 \node (5) at (1.4,1.56) [source] {$5 \sep {\mathbf 1}$};
 \node (6) at (-1.68,-2.12) [source]  {$6\sep {\mathbf 1}$};
 \node (7) at (0,-2.5) [source]  {$7\sep {\mathbf 1}$};
 \node (8) at (1.68,-2.12) [source]  {$8\sep {\mathbf 1}$};
 \node (9) at (3.5,0) [sink] {$9 \sep {\mathbf 1}$};
  \node (Q) at (0,-3.5) {$Q_{1}$};
 \draw [<-] (1) -- (2) ;
 \draw [<-] (1) -- (4) ;
 \draw [<-] (1) -- (6) ;
 \draw [->] (1) -- (9) ;
 \draw [->] (2) -- (3) ;
 \draw [->] (3) -- (9.north) ;
 \draw [->] (4) -- (5) ;
 \draw [->] (5) -- (9) ;
 \draw [->] (6) -- (7) ;
 \draw [->] (7) -- (8) ;
 \draw [->] (8) -- (9.south) ;
 \draw [->,dashed] (9) -- (6) ;
 \draw [->,dashed] (9) -- (4) ;
 \draw [->,dashed,bend right=30] (9) to (2.345) ;
   \draw [leadsto] (4.5,0) -- (6.5,0);
 \node at (5.5,1) {$\mu_{2}$};
\end{tikzpicture}
\begin{tikzpicture}[scale=0.7]
 \node (1) at (-3.5,0) [inner sep=1pt,circle,draw,thick,dashed] {$1\sep {\mathbf 2}$};
 \node (2) at (-1.68,3.08) [sink] {$2\sep {\mathbf r}$};
 \node (3) at (1.68,3.08) [inner sep=1pt,circle,draw,thick,dashed] {$3\sep {\mathbf 1}$};
 \node (4) at (-1.4,1.56) [source] {$4\sep {\mathbf 1}$};
 \node (5) at (1.4,1.56) [source] {$5\sep {\mathbf 1}$};
 \node (6) at (-1.68,-2.12) [source] {$6\sep {\mathbf 1}$};
 \node (7) at (0,-2.5) [source] {$7\sep {\mathbf 1}$};
 \node (8) at (1.68,-2.12) [source] {$8\sep {\mathbf 1}$};
 \node (9) at (3.5,0) [sink] {$9\sep {\mathbf 1}$};
  \node (Q) at (0,-3.5) {$Q_{2}$};
 \draw [->] (1) -- (2) ;
 \draw [<-] (1) -- (4) ;
 \draw [<-] (1) -- (6) ;
 \draw [<-] (2) -- (3) ;
 \draw [->] (4) -- (5) ;
 \draw [->] (5) -- (9) ;
 \draw [->] (6) -- (7) ;
 \draw [->] (7) -- (8) ;
 \draw [->] (8) -- (9.south) ;
 \draw [->,dashed] (9) -- (6) ;
 \draw [->,dashed] (9) -- (4) ;
 \draw [<-,bend right=30] (9) to (2.345) ;
\end{tikzpicture}
}\]

We obtain the graded quiver $Q_{2}$ after mutating at the vertex $2$ in $Q_{1}$. Note that in $Q_2$ we lose information about the rank $\mathbf{r}$ of vertex $2$ and the sink/source property of vertices $1$ and $3$. By using the grading and rank additivity at vertex $2$ in $Q_2$, we recover $\mathbf{r}$ with the equation
\[
  2 \cdot \mathbf{r} =  2 + 1 + 1 = 4,
\]
and thus $\mathbf{r}= 2$. To decide if $1$ is a sink or a source, we follow the proof of Proposition~\ref{prop.uniqueness_sink-source}(b). We observe that the sum of the ranks of the vertices having an arrow from vertex $1$ to them is $2$, hence equal to the rank of vertex $1$. Therefore vertex $1$ in $Q_2$ is a source. Similarly, one can see that vertex $3$ is a source of the mutated quiver.
\end{example}

\section{Mutation of quivers of exceptional sequences}
\label{section.mut_excep_sequences}
 
In this section we introduce the notion of a quiver of an exceptional sequence (Definition~\ref{definition.quiver_except}). Then we develop a mutation rule for these quivers (Proposition~\ref{left.mutation}). The main goal is to be able to perform right and left mutations in $\cohX$ when the only information available on the exceptional sequence is the rank of its objects, and the dimension of the morphism and extension spaces between its objects (Proposition~\ref{E_M.mutation}).

\begin{definition} \label{definition.quiver_except}
Let $\varepsilon = (E_1, \ldots, E_n)$ be an exceptional sequence in a hereditary category $\cH$, and let $a_{i,j} = \dim_{\bbK} \Hom(E_i,E_j) - \dim_{\bbK} \Ext^1(E_i,E_j)$ for $1 \le i < j \le n$. It is not hard to see that for an exceptional pair $(E_i,E_j)$ at least one of $\Hom(E_i,E_j)$ or $\Ext^1(E_i,E_j)$ must be zero. We define the quiver $Q_\varepsilon$ of the exceptional sequence $\varepsilon$ as the quiver having vertices $1, \ldots, n$ and $a_{i,j}$ arrows from the vertex $i$ to the vertex $j$, where a negative number of arrows denotes arrows going in the opposite direction.
\end{definition}

\begin{remark}
Idun Reiten informed us that in joint work with Buan and Thomas \cite{BRT} they are independently introducing the same concept for $\mod H$. They show that in this case the quiver is acyclic.
\end{remark}

Let $\lambda_{\ell} \varepsilon$ be the left mutation at $\ell$ of the exceptional sequence $\varepsilon$, and denote  its quiver by $Q_{\lambda_{\ell}  \varepsilon}$. Then we have the following.

\begin{proposition}[Left mutation rule]\label{left.mutation}
Let $b_{i,j}$ denote the number of arrows from vertex $i$ to vertex $j$ in the quiver $Q_{\lambda_{\ell} \varepsilon}$, for $1 \le i < j \le n$. Then the values $b_{i,j}$ can be obtained from $Q_{\varepsilon}$ by the following equations.
\begin{align*}
  b_{i,j} = & \, a_{i,j} \text{ for } i,j \not\in \{ \ell, \ell+1 \},
\intertext{that is, arrows not involving the vertices $\ell$ or $\ell+1$ do not change,}
  b_{i,\ell+1} = & \, a_{i, \ell} \text{ for } i < \ell \qquad \text{and} \qquad  b_{\ell+1,j} = \, a_{\ell,j} \text{ for } \ell+1 < j
\intertext{that is, arrows involving vertex $\ell$ (and not $\ell+1$) before mutation are moved to vertex $\ell+1$, and}
  b_{i,\ell} = &
    \begin{cases}
          a_{i, \ell} a_{\ell, \ell+1} - a_{i,\ell+1}  & (\bfE) \\
          a_{i ,\ell+1} - a_{i,\ell} a_{\ell,\ell+1}   & (\bfM) \\
          a_{i,\ell+1} + a_{i,\ell} a_{\ell,\ell+1}    & (\bfX) \\           
    \end{cases} \qquad \text{ for } i < \ell \\
  b_{\ell,\ell+1} = &
    \begin{cases}
        a_{\ell,\ell+1} & (\bfE) \\
       -a_{\ell,\ell+1} & (\bfM) \text{ or } (\bfX) \\
    \end{cases} \\
  b_{\ell,j} = & 
    \begin{cases}
          a_{\ell,j} a_{\ell,\ell+1} - a_{\ell+1,j}   & (\bfE) \\
          a_{\ell+1,j} - a_{\ell,j} a_{\ell,\ell+1}   & (\bfM) \\
          a_{\ell+1,j} + a_{\ell,j} a_{\ell,\ell+1}   & (\bfX) \\           
    \end{cases} \qquad \text{ for } \ell+1 < j,
\end{align*}
according to whether the left mutation $\lambda_l$ is of type $(\bfE)$, $(\bfM)$ or $(\bfX)$ (see Section~\ref{Section.Exceptional_sequences}).
\end{proposition}

\begin{proof}
The first two equations follow immediately from the definition of the left mutation, since the object between which we calculate the $\Hom$- and $\Ext$-spaces do not change.

The remaining claims are checked case by case. We suppose that the left mutation $L:= L_{E_{\ell}} E_{\ell + 1}$ is of type $(\bfE)$, the other types can be treated similarly.

In this case we have a short exact sequence $L \into E_{\ell}^{a_{\ell,\ell+1}} \onto E_{\ell+1}$ in $\cohX$. Applying $\Hom(E_i, -)$ to this sequence we obtain the following exact sequence
\begin{align*}
  \Hom(E_i,L) \into \Hom(E_i ,E_{\ell}^{a_{\ell,\ell+1}}) \to \Hom(E_i,E_{\ell+1}) \to \\
   \Ext^1(E_i,L) \to \Ext^1(E_i,E_{\ell}^{a_{\ell,\ell+1}}) \onto \Ext^1(E_i,E_{\ell+1})
\end{align*}
which gives the equation $b_{i, \ell} = a_{i,\ell} a_{\ell,\ell+1} - a_{i,\ell+1}$ for $i < \ell$. The remaining formulas are obtained similarly by applying $\Hom(L, -)$ and $\Hom(-, E_j)$ to the short exact sequence above.
\end{proof}

Let $ \rho_l\varepsilon$ be the right mutation of $\varepsilon$ at vertex $l$. The rule to obtain the quiver $Q_{ \rho_l\varepsilon}$ from $Q_{\varepsilon}$ can be computed in a dual manner.

From the mutation rule above, we note that we do not need to know the exceptional sequence $\varepsilon$ in order to perform left or right mutations on $Q_\varepsilon$. However, we do need to know if the mutation is of type $(\bfE)$, $(\bfM)$ or $(\bfX)$, or a transposition.

\begin{proposition}\label{X.mutation} Let $Q_\varepsilon$ be the quiver of $\varepsilon$.
\begin{enumerate}
\item Assume that $a_{\ell,\ell+1} = 0$ for some vertex $1\le \ell \le n$. Then the left (resp.\ right) mutation at vertex $\ell$ is a transposition, and we can obtain $Q_{\lambda_{\ell} \varepsilon}$ (resp.\ $Q_{\rho_{\ell} \varepsilon}$) by using the left (resp.\ right) mutation rule above.
\item Assume that $a_{\ell,\ell+1}<0$ for some vertex $1\le \ell \le n$. Then the left (resp.\ right) mutation at vertex $\ell$ is of type $(\bfX)$, and we can obtain $Q_{\lambda_{\ell} \varepsilon}$ (resp.\ $Q_{\rho_{\ell} \varepsilon}$) by using the left (resp.\ right) mutation rule above.
\item Assume that $a_{\ell,\ell+1}>0$ for some vertex $1\le \ell \le n$. Then the left (resp.\ right) mutation at vertex $\ell$ is either of type $(\bfE)$ or of type $(\bfM)$. Hence we cannot obtain $Q_{\lambda_{\ell} \varepsilon}$ (resp.\ $Q_{\rho_{\ell} \varepsilon}$) by using the left (resp.\ right) mutation rule above without further information.
\end{enumerate}
\end{proposition}

\begin{proof}
Observe that the mutations are transpositions if and only if $\Hom(E_{\ell}, E_{\ell+1}) = \Ext^1(E_{\ell}, E_{\ell+1}) = 0$, and this is the case if and only if $a_{\ell, \ell+1} = 0$. This proves (a).

Similarly the left (resp.\ right) mutation of type $(\bfX)$ occurs if and only if $\Ext^1(E_l,E_{l+1})\ne 0$. This last condition is equivalent to $a_{l,l+1}<0$, thus we have (b).

In these two cases application of the rule is then straightforward.

In the remaining cases the mutations are of type $(\bfE)$ or $(\bfM)$.
\end{proof}

Now we assume $\cH = \cohX$ for some weighted projective line $\bbX$.

\begin{proposition} \label{E_M.mutation}
Let $\rank(E_1), \ldots, \rank(E_n)$ be the ranks of the objects in $\varepsilon$. If $\rank(E_{\ell}) > 0$ or $\rank(E_{\ell+1}) > 0$, then we can perform the left (resp.\ right) mutation at vertex $\ell$ in $Q_\varepsilon$.

Moreover, we can calculate the ranks of the objects after mutation from the given ranks.
\end{proposition}

\begin{proof}
By Proposition~\ref{X.mutation} it only remains to distinguish the types $(\bfE)$ and $(\bfM)$ when $a_{\ell, \ell+1} > 0$. In this setup the left mutation is of type $(\bfE)$ (resp.\ $(\bfM)$) if and only if $\rank(E_{\ell}) a_{\ell,\ell+1} > \rank(E_{\ell+1})$ (resp.\ $\rank(E_{\ell}) a_{\ell,\ell+1} \le \rank(E_{\ell+1})$). Finally, the rank of the mutated object can be calculated by using the short exact sequence defining the left (resp. right) mutation.
\end{proof}

Observe that if $\rank(E_{\ell}) = \rank(E_{\ell+1}) = 0$, we cannot distinguish between the cases $(\bfE)$ or $(\bfM)$ in general. This is due to the fact that we cannot determine if the approximation defining the left (resp.\ right) mutation is an epimorphism or a monomorphism. Thus we cannot apply the mutation rule for quivers of exceptional sequences.

\section{Recovering the tilting object}\label{section.recoverT}

In this section we focus on the following problem: Let $T$ be a cluster-tilting object in $\cC:=\cC_\bbX$, the cluster category of $\cohX$, and denote by $\Gamma=\End_\cC(T)$ its endomorphism ring. Suppose that one is given $Q_\Gamma$, the quiver of $\Gamma$ equipped with the natural grading, and the ranks of the indecomposable summands of $T$. Can we recover $T$ and $\Gamma$?

We give the following positive answer to this question:

\begin{theorem} \label{theorem.reconstruct_from_Q+rk}
Let $T$ be a cluster-tilting object in $\cC:=\cC_\bbX$, the cluster category of $\cohX$, and denote by $\Gamma=\End_\cC(T)$ its endomorphism ring.
 
Assume that we are given $Q_\Gamma$, the quiver of $\Gamma$ equipped with the natural grading, and the ranks of the indecomposable summands of $T$. Then $T$ and $\Gamma$ are determined uniquely up to twist with a line bundle and choice of a parameter sequence $\lambdab$.

Moreover they can be determined algorithmically.
\end{theorem}

\begin{remarks}\mbox{}
\begin{enumerate} 
\item The uniqueness claimed by the theorem is ``as good as possible'', that is, if we twist by a line bundle or change the parameter sequence then neither the graded quiver $Q_{\Gamma}$ nor the ranks of the summands of $T$ change (see \cite[Section~4.4]{Meltzer2}).
\item Choice of a parameter sequence $\lambdab$ is more than choice of a weighted projective line: If several parameters have the same weight we may interchange them in the parameter sequence without changing the weighted projective line, but this will possibly give different choices for $T$ (See Example~\ref{ex.changing_arms} and \cite[Section~2]{LM3}).
\end{enumerate}
\end{remarks}

Our strategy for the proof of Theorem~\ref{theorem.reconstruct_from_Q+rk} is as follows: we provide a method to transform any cluster-tilting object $T$ into a well-known cluster-tilting object in $\cC$, the squid (see Definition~\ref{Squid}). Since the quiver of the cluster-tilting objects in $\cC$ is not known to be connected in the wild case (and even if it is, there is no algorithm to find the connecting mutation sequence), we cannot use (cluster-tilting) mutation directly. We will however see that we can use mutation of exceptional sequences and their quivers (see Definition~\ref{definition.quiver_except}). Then we will use the fact that the cluster-tilting object of the squid is unique up to a twist with line bundles and interchanging arms of the same length.

This section is divided as follows: In Subsection~\ref{subsect.prep_for_algorithm} we prove some preliminary results, which will be used in the actual algorithm given in Subsection~\ref{back.to.squid.section}. There we explain how to use mutation of exceptional sequences to transform any given tilting sequence to the tilting sequence of the squid. Finally, in Subsection~\ref{recover.T.section}, we complete the proof of Theorem~\ref{theorem.reconstruct_from_Q+rk}.

\subsection{Preliminaries} \label{subsect.prep_for_algorithm}

First, in Subsection~\ref{subsection.rank_reduction}, we recall a useful theorem from \cite{Meltzer} which provides a rank reduction technique. Then, in Subsection~\ref{subsection.perp_cats}, we provide a result on perpendicular categories of line bundles. Finally, in Subsection~\ref{subsub.rem_hom}, we show how to remove $\Hom$-arrows in the setup of Subsection~\ref{subsection.perp_cats}.

\subsubsection{Rank reduction} \label{subsection.rank_reduction}
We recall a theorem of Meltzer \cite{Meltzer}, which is one of the key ingredients in our method to recover the tilting object from certain combinatorial data.

For an exceptional sequence  $\varepsilon=(E_1,\ldots,E_n)$ in $\cohX$ we define
\[
 \| \varepsilon \| = \left(\rank(E_{\pi(1)}) ,
   \ldots,\rank(E_{\pi(n)}) \right),
\] 
where $\pi$ is a permutation of $1,\ldots,n$ such that $\rank(E_{\pi(1)}) \ge \cdots \ge \rank(E_{\pi(n)})$. For sequences $e=(e_1,\ldots,e_n)$ and $f=(f_1,\ldots,f_n)$ in $\bbN_0^n$, we write $e \le f$ if $e_i \le f_i$ for all $1\le i \le n$, and $e<f$ provided that $e\le f$ and $e\ne f$.

\begin{theorem}[{Rank reduction -- \cite[Propositon 3.2]{Meltzer}}] \label{rank.reduction}
 Let $\varepsilon = (E_1, \ldots, E_n)$ be a complete exceptional sequence in $\cohX$ such that there exist integers $a < b$ with the properties:
\begin{enumerate}
\item $\rank E_a \ge 2$ and $\rank E_b \ge 2$,
\item $\Hom (E_a,E_b) \neq 0$ and $\Hom (E_i,E_j)=0$ for $a \le i < j \le b$ with $(i,j) \neq (a,b)$.
\end{enumerate}
For
\[ g = 
  \begin{cases}
    \lambda_{b-1}\lambda_{b-2} \cdots \lambda_{a} & \text{ if any nonzero map } E_a \to E_b \text{ is a monomorphism} \\
    \rho_{a}\rho_{a+1} \cdots \rho_{b-1} & \text{ if any nonzero map } E_a \to E_b \text{ is an epimorphism}
  \end{cases}
\]
we have $\| g\varepsilon \| < \|\varepsilon \|$.
\end{theorem}

\begin{remarks} \mbox{}
\begin{enumerate}
\item Observe that $\Ext^1(E_b,E_a)=0$, and thus by Lemma~\ref{lemma.HR}, either all nonzero morphisms $E_a \to E_b$ are monomorphisms, or all such morphisms are epimorphisms. Moreover, they are monomorphisms (resp.\ epimorphisms) if and only if $\rank(E_a) \le \rank(E_b)$ (resp.\ $\rank(E_a)> \rank(E_b)$).
\item The proof of the theorem uses Lemma~\ref{Hom.transposition} to show that all the left (resp.\ right) mutations except for $\lambda_{b-1}$ (resp.\ $\rho_a$) are transpositions. So we only apply  one ``real'' mutation in each case.
\end{enumerate}
\end{remarks}

\subsubsection{Perpendicular categories} \label{subsection.perp_cats}

We now recall a reduction technique that allows us to use the tools developed for exceptional sequences in hereditary Artin algebras in the case of $\cohX$.

We denote by $\cT$ the full subcategory of $\cohX$ of all objects $X$ satisfying $ \Hom(X,\cO(\vc))=0$. Observe that all finite length sheaves are contained in $\cT$. Further let $\cF = \{Y \in \vect \bbX | \Ext^1(Y,\cO)=0 \}$. Let $T=\oplus_{0\le\vx \le \vc} \, \cO(\vx)$ be the canonical tilting bundle in $\cohX$ and denote by $\Lambda=\End_{\cohX}(T)$ the corresponding canonical algebra.
 
Note that there are monomorphisms $\cO(\vx) \into \cO(\vc)$ for $0\le \vx \le \vc$ which induce monomorphisms $\Hom(X,\cO(\vx)) \into \Hom(X,\cO(\vc))$. Thus
\[ \cT = \{X \in \cohX | \Hom(X,T)=0\}.\]
Similarly, the monomorphisms $\cO \into \cO(\vx)$  induce epimorphisms $\Ext^1(Y,\cO)\onto \Ext^1(Y,\cO(\vx))$ for $0 < \vx \le \vc$ whenever $Y \in \vect \bbX$ and thus
\[ \cF =\{Y \in \vect \bbX | \Ext^1(Y,T)=0\} = \{Y \in \cohX | \Ext^1(Y,T)=0\}. \]
In particular, we have that $\cT\cap \cF=\{0\}.$

\begin{theorem}\label{Oc.left.perp}
The left perpendicular category 
\[\lperp{\cO}(\vc) = \{X \in \cohX ~|~\Hom(X,\cO(\vc)) = 0 = \Ext^1(X,\cO(\vc)) \}\] 
can be identified with the module category  $\mod \Lambda_{\vc}$ where $\Lambda_{\vc}$ is the hereditary algebra $\Lambda/e_{\vc} \Lambda e_{\vc}$. Under this identification, the line bundle $\cO(2\vc)$ and the torsion sheaves $S_{\vx}^{\vc}$  ($0< \vx < \vc$) form a complete system of indecomposable injective modules in $\mod \Lambda_{\vc}$, where $S_{\vx}^{\vc}$ is the cokernel of the (up to scalars unique) monomorphism $\cO(\vx) \into \cO(\vc)$.
\end{theorem}

\begin{proof}
 We follow the ideas of the proof of \cite[Proposition 3.4]{LP}. The functor $\D\R\Hom(-,T)$ is an equivalence from $D^b(\cohX)$ to $D^b(\mod \Lambda)$. Denote by 
\begin{align*}
\cX & = \D \R \Hom(\cT, T)[1] = \D \Ext^1(\cT, T) \text{, and} \\
\cY & = \D \R \Hom(\cF, T) = \D \Hom(\cF, T)
\end{align*}
the images of $\cT$ and $\cF$ under this equivalence (shifted to the module category of $\Lambda$), respectively. Moreover, denote by $I_{\vx}$ the indecomposable injective $\Lambda$-module at $\vx$ for $0\le \vx \le \vc$ and by $I_{\vx}^{\vc}$ the corresponding indecomposable injective $\Lambda_{\vc}\,$-modules.  Then we have that
\begin{align*}
 \lperp{\cO}(\vc) &= \{X \in \cT \mid \Ext^1(X,\cO(\vc))=0 \} \\
 & = \{X \in \cT \mid \Hom(\D\R\Hom(X, T), \underbrace{\D\R\Hom(\cO(\vc), T)}_{= I_{\vc}}) = 0 \} \\
 & = \{ M \in \cX \mid \Hom(M,I_{\vc})=0\}.
\end{align*}
Now observe that
\[ \{M \in \cY \mid \Hom(M,I_{\vc})=0 \} = \{Y \in \cF \mid \Hom(Y, \cO(\vc)) = 0\} = \cF \cap \cT = 0. \]
Therefore, since any $\Lambda$-module is the extension of one in $\cX$ and one in $\cY$, we have that 
\begin{align*}
\lperp{\cO}(\vc) &= \{ M \in \cX \mid \Hom(M, I_{\vc}) = 0 \} \\
& = \{ M \in \mod \Lambda \mid \Hom(M,I_{\vc})=0\} \\
& = \mod \Lambda_{\vc}.
\end{align*}
Now for $0 < \vx < \vc$ we have short exact sequences $ I_{\vx}^{\vc} \into I_{\vx} \onto I_{\vc}$. Thus, in the derived category we have
\[ I_{\vx}^{\vc} = \Cone(I_{\vx} \to I_{\vc})[-1] = \Cone(\cO(\vx) \to \cO(\vc)) = S_{\vx}^{\vc} \]
(note that the shift vanishes since we identify via $\D\R\Hom(-,T)[1]$), and similarly
\[ I_0^{\vc} = \Cone(I_0 \to I_{\vc}^2)[-1] = \Cone(\cO \to \cO(\vc)^2) = \cO(2\vc). \qedhere \]
\end{proof}

\subsubsection{Removing $\Hom$-arrows} \label{subsub.rem_hom}

We now assume to be in the following situation: Let $(E_1, \ldots, E_n)$ be an exceptional sequence with $E_1 = \cO(\vc)$. Then $(E_2, \ldots, E_n)$ is an exceptional sequence in $\lperp{\cO}(\vc) = \mod \Lambda_{\vc}$.

When we talk about the ranks of the objects in $\mod \Lambda_{\vc}$, we consider them in $\cohX$.

\begin{lemma} \label{back.to.orthogonal_explicit}
Let $\epsilon = (E_2, \ldots, E_n)$ be an exceptional sequence in $\mod \Lambda_\vc$, and $\ell$ such that $\dim_{\bbK} \Hom(E_{\ell}, E_{\ell+1}) = m > 0$. Then there is $t \in \mathbb{Z}$ such that the sequence
\[ \rho_{\ell}^t \epsilon = (E_2, \ldots, E_{\ell-1}, E_{\ell}', E_{\ell+1}', E_{\ell+2}, \ldots, E_n) \]
(here we write $\rho_{\ell}^{-1} = \lambda_{\ell}$) is such that $\Hom(E_{\ell}', E_{\ell+1}') = 0$. Furthermore,
\begin{enumerate}
\item if $m = 1$ then there is a unique such $t$ in $\{0, \pm1\}$,
\item if $m > 1$ and $\rank E_{\ell} \leq \rank E_{\ell+1}$, then $t$ is unique, and it is the minimal nonpositive integer such that the exceptional pair $\rho^{t+1}_1(E_{\ell}, E_{\ell+1}) = (X,Y)$ satisfies $m \rank(X) \le \rank(Y)$, and
 \item if $m > 1$ and $\rank E_{\ell} \geq \rank E_{\ell+1}$, then $t$ is unique, and it is the minimal nonnegative integer such that the exceptional pair $\rho^{t-1}_1(E_{\ell}, E_{\ell+1})= (X,Y)$ satisfies $\rank(X) > m \rank(Y)$.
\end{enumerate}
\end{lemma}

\begin{proof}
Let $C(E_{\ell},E_{\ell+1})$ be the smallest full subcategory of $\mod \Lambda_\vc$ containing $E_{\ell}$ and $E_{\ell+1}$ and that is closed under extensions, kernels of epimorphisms, and cokernels of monomorphisms. By \cite[Theorem~5]{CB}, the category $C(E_{\ell},Z_{\ell+1})$ is equivalent with the category of representations of the generalized Kronecker quiver with $m$ arrows. By iterated application of $L$ (or $R$), we reach the exceptional pair having nonzero extensions. This is the pair consisting of the simple objects  in $C(E_{\ell}, E_{\ell+1})$. Thus we find the integer $t$ by following the (right or left) mutations that decrease the ranks.
\end{proof}

\subsection{The algorithm} \label{back.to.squid.section}

In this subsection, we illustrate a method to obtain a sequence of (left and right) mutations in order to transform a given exceptional sequence into the exceptional sequence of the squid in $\cohX$.

The algorithm consist of the following steps: We first mutate in order to obtain an exceptional sequence where the first term is a line bundle. This is possible since any complete exceptional sequence of length $n$ in $\cohX$ can have at most $n-2$ torsion objects. Then we reduce $\Hom$-arrows between the remaining vector bundles of the exceptional sequence. Next we remove $\Ext$-arrows between torsion sheaves. Finally we remove all sources among the torsion sheaves.

We start with an arbitrary complete exceptional sequence $(E_1, \ldots, E_n)$ in $\cohX$. By abuse of notation we also call the updated exceptional sequence at any stage of the algorithm $(E_1, \ldots, E_n)$.

\begin{algorithm}[Step 1: Obtaining a line bundle] \label{alg.step_1} $ $
\begin{enumerate}
\item Preparation. If there are torsion sheaves in the exceptional sequence, then move them all the way to the right of the vector bundles using left mutations.
\item Reduction. Choose $a < b$ such that $E_a$ and $E_b$ are vector bundles with $\Hom(E_a, E_b) \neq 0$, and $\Hom(E_i, E_j) = 0$ for any $a \leq i < j \leq b$ with $(i,j) \neq (a,b)$. Apply Theorem~\ref{rank.reduction} (rank reduction).
\item Iterate until the exceptional sequence contains at least one line bundle. Then move this line bundle to the first position using right mutations.
\end{enumerate}
\end{algorithm}

Now we have an exceptional sequence where the first term is a line bundle. Up to twist with a line bundle we may assume $E_1 = \cO(\vc)$.

\begin{algorithm}[Step 2 -- $\Hom$-reduction] \label{alg.step_2} $ $
\begin{enumerate}
\item Preparation. If there are torsion sheaves in the exceptional sequence, then move them all the way to the right of the vector bundles using left mutations as in Algorithm~\ref{alg.step_1}(a).
\item Reduction. Choose $1 < a < b$ such that $E_a$ and $E_b$ are vector bundles with $\Hom(E_a, E_b) \neq 0$, but $\Hom(E_i, E_j) = 0$ for any other $a \leq i < j \leq b$. We can make $E_a$ and $E_b$ adjacent by using transpositions (see çLemma \ref{Hom.transposition}), thus we may assume $b = a+1$. Using Lemma~\ref{back.to.orthogonal_explicit}, find $t \in \bbZ$ such that 
\[ (E_1, \ldots, E_{a-1}, E_a^{\rm new}, E_b^{\rm new}, E_{b+1}, \ldots, E_n) = \rho_a^t (E_1, \ldots, E_n) \]
with $\Hom(E_a^{\rm new}, E_b^{\rm new}) = 0$.
\item Iterate until no more $\Hom$-arrows exist between objects of positive rank (not counting $E_1 = \cO(\vc)$).
\end{enumerate}
\end{algorithm}

Note that the iteration of Algorithm~\ref{alg.step_2} stops after a finite number of steps, since the number of torsion sheaves in the exceptional sequence is weakly increasing. When this number does not change, the proper reduction performed in (b) decreases the sum of the lengths of the objects in the exceptional sequence, where the length is  considered in $\mod \Lambda_\vc$ (see \cite[Section 6]{Ringel3}).

At this point our exceptional sequence has the following shape:
\[ (\underbrace{E_1}_{\cO(\vc)}, \underbrace{E_2, \ldots, E_m}_{\substack{\text{orthogonal} \\ \text{vector bundles}  }}, \underbrace{E_{m+1}, \ldots, E_n}_{\text{torsion sheaves}}) \]
We now show more precisely the following:

\begin{proposition} \label{prop.m_=_2}
In the situation above we have $m = 2$, and $E_2 = \cO(2\vc)$.
\end{proposition}

\begin{proof}
 Let $C(E_2, \ldots, E_m)$ be the smallest full subcategory  of $\mod \Lambda_{\vc}$ which contains $E_2, \ldots, E_m$ and is closed under extensions, kernels of epimorphisms and cokernels of monomorphisms. By \cite[Theorem~5]{CB}, the category $C(E_2, \ldots, E_m)$ is equivalent to the category of representations of a quiver $Q_{m-1}$ with $m-1$ vertices and no oriented cycles. Since the exceptional sequence $(E_2, \ldots, E_m)$ is orthogonal in $\mod Q_{m-1}$, it is an exceptional sequence of the simple $\bbK Q_{m-1}$-modules.
 
Let $\sigma$ be as in Lemma~\ref{from.simples.to.injectives.lemma}. Then 
\[ \sigma \cdot (E_2, \ldots, E_m) = (I_m, \ldots, I_2) \]
where $I_i$ is the injective envelope of the simple $\bbK Q_{m-1}$-module $E_i$ for $2 \leq i \leq m$.

Recall that the sequence $\sigma$ is just made up of left mutations of type $(\bfX)$. Thus the $I_i$, seen as objects in $\cohX$, are vector bundles.

 We now consider the sequence $(E_{m+1},\ldots, E_n)$ of rank zero objects. Observe that $(E_{m+1},\ldots, E_n\rperp{)}$, as a full subcategory of $\cohX$, is equivalent to the category of coherent sheaves over a weighted projective line $\bbX'$. Let  $w_\bbX$ and $w_{\bbX'}$ be the weight sequences corresponding to $\bbX$ and $\bbX'$, respectively. Then $w_\bbX$ dominates $w_{\bbX'}$, i.e.\ $w_\bbX \ge w_{\bbX'}$. By iterated use of \cite[Proof of Theorem~9.8]{GL}, the category $\cohX'$ is equivalent to a full (exact) subcategory of $\cohX$ which is closed under extensions, and $\cO_{\bbX'}(\vc)$ can be identified with $\cO_\bbX (\vc)$. Moreover, the inclusion $\coh\bbX' \subset \cohX$ preserves ranks. So in particular, we have that $\vect{\bbX'}= \vect{\bbX}\cap\coh\bbX'$ and $\coh_0\bbX'=\coh_0\bbX \cap \coh \bbX'$. 

We look at the sequence $(\cO(\vc), I_m, \ldots, I_2)$ as an exceptional sequence in $\coh\bbX'$. By applying Theorem~\ref{Oc.left.perp} for $\coh\bbX'$, we see that $I_m \simeq \cO(2\vc)$ and the $I_i$ are torsion sheaves in $\cohX'$ for $2 \le i < m$ (i.e.\ rank zero objects). By the observation above that all the $I_i$ are vector bundles this means $m=2$.
\end{proof}

Thus our exceptional sequence is of the form
\[ (\underbrace{E_1}_{\cO(\vc)}, \underbrace{E_2}_{\cO(2\vc)}, \underbrace{E_3, \ldots, E_n}_{\text{torsion sheaves}}). \]

\begin{algorithm}[Step 3 -- $\Ext$-reduction] \label{alg.step_3} $ $
\begin{enumerate}
\item Preparation. Choose $a < b$ such that $\Ext^1(E_a, E_b) \neq 0$ with $b-a$ minimal. We apply transpositions to move $E_a$ and $E_b$ next to each other (this is possible by iterated use of Lemma~\ref{moving.closer.lemma}).
\item Reduction. Apply the right mutation $\rho_{b-1} = \rho_a$ to reduce the nonzero $\Ext$ groups of the exceptional sequence by one. Here the right mutation is of type $(\bfX)$.
\item Iterate until no more $\Ext$-arrows exist. 
\end{enumerate}
\end{algorithm}

Observe that the above process stops after a finite number of iterations (see the proof of Theorem~\ref{Ext.arrow.reduction}).

Note also that there are no extensions from vector bundles to torsion sheaves, or from $\cO(\vc)$ to $\cO(2\vc)$. Therefore $E_1$ and $E_2$ are not affected by Algorithm~\ref{alg.step_3}, whence we now have a tilting sequence of the form
\[ (\underbrace{E_1}_{\cO(\vc)}, \underbrace{E_2}_{\cO(2\vc)}, \underbrace{E_3, \ldots, E_n}_{\text{torsion sheaves}}). \]

\begin{algorithm}[Step 4 -- removing superfluous sources] \label{alg.step_4} $ $
\begin{enumerate}
\item Mutation. If there is a torsion sheaf $E_a$ which is a source in the quiver of the exceptional sequence, then apply tilting mutation (that is, right mutate it past all elements in the exceptional sequence to which there are irreducible morphisms from $E_a$, see  \cite[Proposition~2.3]{Hubner2}).
\item Iterate until there are no more sources among the torsion sheaves.
\end{enumerate}
\end{algorithm}

After this final algorithm, we have a tilting sequence of the form
\[ (\underbrace{E_1}_{\cO(\vc)}, \underbrace{E_2}_{\cO(2\vc)}, \underbrace{E_3, \ldots, E_n}_{\substack{\text{torsion sheaves,} \\ \text{no sources}}}). \]
Thus we have shown the following.

\begin{theorem}\label{back.to.squid.thm}
Let $(E_1,\ldots,E_n)$ be an exceptional sequence in $\cohX$. Then, applying Algorithms~\ref{alg.step_1}, \ref{alg.step_2}, \ref{alg.step_3}, and \ref{alg.step_4}, one obtains a sequence $\alpha$ of (left and right) mutations such that $\alpha \cdot (E_1, \ldots, E_n)$ is an exceptional sequence of the squid.
\end{theorem}

\subsection{Recovering $T$.} \label{recover.T.section} We have all the ingredients to complete the proof of the main theorem of this section.

\begin{proof}[Proof of Theorem~\ref{theorem.reconstruct_from_Q+rk}]
Using the graded quiver $Q_{\Gamma}$, we can recover $Q_{(E_1, \ldots, E_n)}$, the quiver of a tilting sequence $(E_1, \ldots, E_n)$ in $\cohX$ given by a tilting sheaf corresponding to $Q_{\Gamma}$. This quiver is unique up to transpositions. Using the algorithms discussed in Subsection~\ref{back.to.squid.section}, we obtain a sequence of mutations taking us to the squid (Theorem~\ref{back.to.squid.thm}). Note that we can keep track of the quiver and ranks throughout this procedure: In Algorithms~\ref{alg.step_1}, \ref{alg.step_2}, and \ref{alg.step_3} we do not mutate two torsion sheaves past each other, so by Proposition~\ref{E_M.mutation} we know the quiver and the ranks after mutation. In Algorithm~\ref{alg.step_4} we only do tilting mutations in sources, so we can also keep track of the quivers here (and the ranks of all the objects affected by Algorithm~\ref{alg.step_4} are $0$).

Now the cluster-tilting object $T_{\sq}(\vc)$ associated to the line bundle $\cO(\vc)$ is uniquely determined up to choice of a parameter sequence $\lambdab$ by the quiver, and the fact that it consists of $\cO(\vc)$, $\cO(2\vc)$, and torsion sheaves.
 
After making a choice for the parameter sequence $\lambdab$, we can apply the inverse sequence of mutations obtained so far to obtain a tilting sheaf which has the graded quiver and ranks we started with. The only choices made along the way are the choice of a parameter sequence $\lambdab$, and the fact that we randomly set a line bundle to be $\cO(\vc)$ between Algorithms~\ref{alg.step_1} and \ref{alg.step_2}. Thus the tilting sheaf we reconstructed is as unique as claimed by the theorem.

Clearly now we can also calculate $\Gamma$.
\end{proof}

As a consequence we obtain the following.

\begin{corollary}

Let $T_1$ and $T_2$ be two cluster-tilting objects in $\cC_{\bbX}$. Assume that $Q_{\Gamma_1} \simeq Q_{\Gamma_2}$, where $\Gamma_i=\End_\cC(T_i)$ for $i=1,2$, and this isomorphism of quivers respects the ranks of the indecomposable summands corresponding to the vertices. Then, regarding $T_1$ and $T_2$ as exceptional sequences, we have
\[T_1 = \alpha^{-1} \circ \phi \circ \alpha \cdot T_2(\vx), \]
where $\alpha$ is a sequence of mutations of exceptional sequences such that $\alpha T_2$ is a squid, $\phi$ is a permutation of the labels of arms of the squid which have the same length, and $\cO(\vx)$ is some line bundle.

\end{corollary}

The following example illustrates the necessity of the permutation $\phi$ in the corollary above.

\begin{example} \label{ex.changing_arms}
For $\lambdab = (\lambda_1, \lambda_2, \lambda_3, \lambda_4) \in (\bbP^1)^4$, consider the quasitilted algebras $\Lambda(\lambdab)$ of tubular type $(2,2,2,2)$ given by the quiver
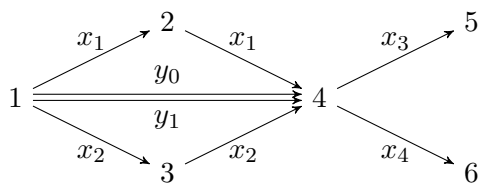
\[ \begin{tikzpicture}[xscale=2]
 \node (1) at (0,0) {$1$};
 \node (2) at (1,1) {$2$};
 \node (3) at (1,-1) {$3$};
 \node (4) at (2,0) {$4$};
 \node (5) at (3,1) {$5$};
 \node (6) at (3,-1) {$6$};
 \draw [->] (1) -- node [above] {$x_1$} (2);
 \draw [->] (1) -- node [below] {$x_2$} (3);
 \draw [->] (2) -- node [above] {$x_1$} (4);
 \draw [->] (3) -- node [below] {$x_2$} (4);
 \draw [->] (1.10) -- node [above] {$y_0$} (4.170);
 \draw [->] (1.350) -- node [below] {$y_1$} (4.190);
 \draw [->] (4) -- node [above] {$x_3$} (5);
 \draw [->] (4) -- node [below] {$x_4$} (6);
\end{tikzpicture} \]
subject to the relations $x_1^2 = \lambda_1^0 y_0 + \lambda_1^1 y_1$, $x_2^2 = \lambda_2^0 y_0 + \lambda_2^1 y_1$, $(\lambda_3^0 y_0 + \lambda_3^1 y_1) x_3 = 0$, and $(\lambda_4^0 y_0 + \lambda_4^1 y_1) x_4 = 0$. (Note that the arrows $y_0$ and $y_1$ are superfluous, but we use them to get a more immediate connection from the parameter sequence $\lambdab$ of the weighted projective line and the relations of the algebra $\Lambda(\lambdab)$.) It is easy to see that the two algebras $\Lambda(\lambdab)$ and $\Lambda(\tilde \lambdab)$ are isomorphic if and only if $\lambdab$ and $\tilde \lambdab$ lie in the same orbit of the action of ${\rm PSL}(2) \times C_2 \times C_2$ on $(\bbP^1)^4$. Here ${\rm PSL}(2)$ acts on the individual components, the first cyclic group acts by interchanging the first two parameters, and the second cyclic group acts by interchanging the latter two parameters.

Now note that the weighted projective lines of types $(2,2,2,2; \lambdab)$ and $(2,2,2,2; \tilde \lambdab)$ have equivalent categories of coherent sheaves if any only if $\lambdab$ and $\tilde \lambdab$ lie in the same ${\rm PSL}(2)$ orbit up to reordering their entries, or equivalently, if they lie in the same ${\rm PSL}(2) \times \Sigma_4$-orbit, where $\Sigma_4$ denotes the symmetric group on four symbols.

It follows that coherent sheaves on the weighted projective lines of types $(2,2,2,2; (1\sep0), (1\sep1), (0\sep1), (\lambda^0\sep \lambda^1))$ and $(2,2,2,2; (1\sep0), (\lambda^0\sep \lambda^1), (0\sep1), (1\sep1))$ are equivalent. But one easily checks that the algebras $\Lambda((1\sep0), (1\sep1), (0\sep1), (\lambda^0\sep \lambda^1))$ and $\Lambda((1\sep0), (\lambda^0\sep \lambda^1), (0\sep1), (1\sep1))$ are typically not isomorphic.
\end{example}

\section{Conclusions and open questions} \label{lastsection}

In the following we summarize some natural questions on what information the quiver of a cluster-tilted algebra contains, and to what extent we know the answers: Throughout $T$ is a cluster-tilting object in $\cC_{*}$, and we denote by $Q$ the quiver of the corresponding cluster-tilted algebra.
\begin{enumerate}  
 \item Given $Q$, is it possible to decide if $\cC_{*} = \cC_{H}$ or $\cC_{*} = \cC_{\bbX}$? 
 \begin{enumerate}
 \item If there exists an acyclic quiver in the mutation class of $Q$, then $\cC_{*} = \cC_{H}$. Otherwise $\cC_{*} = \cC_{\bbX}$. However, note that we do not have an algorithmic way of finding out if such a quiver exists in the mutation class.
 \item Assume additionally that  we know the grading of the quiver $Q$. Then we can first recover the Cartan matrix of the underlying quasitilted algebra. Second, we calculate the Coxeter matrix. Then, from the roots of the Coxeter polynomial, we can recover if $\cC_{*} = \cC_{H}$ or $\cC_{*} = \cC_{\bbX}$ by using \cite[Proposition~9.1]{Lenzing2}. Furthermore, for the case $\cC_{*} = \cC_{\bbX}$ we even recover the weight sequence $\mathbf p$.
\end{enumerate}
 
\item If additionally we know that $\cC_{*} = \cC_{\bbX}$ for some weighted projective line $\bbX = (\bbP^1, \lambdab, \mathbf{p})$, can we recover the weight sequence $\mathbf{p}$?
 \begin{enumerate}
  \item For the Euclidean or tubular case, there exists a sequence of mutations taking $Q$ to the squid. Then one can read off the weight sequence $\mathbf p$ from the arms of the squid. For the wild case, this does not work in general, since it is not known if there is more than one mutation component.
  \item Assuming one has the graded quiver, one can calculate the roots of the Coxeter polynomial and use \cite[Proposition~9.1]{Lenzing2} as before. 
  \item An alternative way, which also uses the graded quiver, is by Theorem~\ref{theorem.reconstruct_from_Q+rk}. One finds a sequence of mutations of exceptional sequences that takes $Q$ to the squid.
   \end{enumerate}

\item In the setup of (b), assume we know $\mathbf{p}$. Can we recover $\lambdab$?
\begin{itemize}
\item[]
 No, we have to choose $\lambdab$, since exceptional sequences are independent of $\lambdab$ by \cite[Section~4.4]{Meltzer2}. In particular, tilting objects are independent of the parameter sequence.
\end{itemize}
   
 \item Given $Q$, how many different sink-source distributions can there be?
 \begin{enumerate}
  \item In case $\cC_{*} = \cC_{H}$, the sink-source distribution is unique by Proposition~\ref{prop.uniqueness_sink-source}(a).
  \item In case $\cC_{*} = \cC_{\bbX}$ for some $\bbX$ not of tubular type, if we additionally know the grading on $Q$, then the sink-source distribution can be uniquely calculated by Proposition~\ref{prop.uniqueness_sink-source}(b).
  \item In case $\cC_{*} = \cC_{\bbX}$ for some $\bbX$ of tubular type, we have seen in Example~\ref{example.several_sink_source_dist} that the sink-source distribution is not even uniquely determined by the graded quiver. It is an open question how many different sink-source distributions there can be.
 \end{enumerate}
 \item Given the (ungraded) quiver $Q$, can one recover the grading on $Q$? 
  \begin{enumerate}
   \item For the case $\cC_{*} = \cC_{H}$, proceed as in the proof of Proposition~\ref{prop.uniqueness_sink-source}.
   \item For the case $\cC_{*} = \cC_{\bbX}$ not wild, find a sequence of mutations taking $Q$ to the squid. Since we know the grading of the squid, use the graded mutation rule to go back to $Q$. The wild case is unknown.
  \end{enumerate}
 \item In case $\cC_{*} = \cC_{\bbX}$, is there a concrete algorithm to recover the grading on $Q$ by using the ungraded quiver and the ranks?

 \begin{enumerate}
  \item This question is open in general.
  \item In examples it is usually quite easy to recover the grading.
 \end{enumerate}
 \end{enumerate}


\begin{thebibliography}{BMRRT}

\bibitem[ABS]{ABS} I. Assem, T. Br{\"u}stle and R. Schiffler.
\newblock Cluster-tilted algebras as trivial extensions.
\newblock {\em Bull. Lond. Math. Soc.} 40 (1), 151--162, 2008.

\bibitem[ARS]{ARS} M. Auslander, I. Reiten, S. Smal\o.
\newblock Representation theory of artin algebras.
\newblock Cambridge University Press 1995.

\bibitem[BKL]{BKL} M. Barot, D. Kussin and H. Lenzing.
\newblock The cluster category of a canonical algebra.
\newblock {\em Trans. Amer. Math. Soc.}, 362 (8), 4313--4330, 2010.

\bibitem[BMR1]{BMR1}
Aslak~Bakke Buan, Robert~J. Marsh, and Idun Reiten.
\newblock Cluster-tilted algebras.
\newblock {\em Trans. Amer. Math. Soc.}, 359(1):323--332 (electronic), 2007.

\bibitem[BMR2]{BMR2}
Aslak~Bakke Buan, Robert~J. Marsh, and Idun Reiten.
\newblock Cluster mutation via quiver representations.
\newblock {\em Comment. Math. Helv.}, 83(1):143--177, 2008.

\bibitem[BMRRT]{BMRRT}
Aslak~Bakke Buan, Robert Marsh, Markus Reineke, Idun Reiten, and Gordana
  Todorov.
\newblock Tilting theory and cluster combinatorics.
\newblock {\em Adv. Math.}, 204(2):572--618, 2006.

\bibitem[BR]{BR}
Aslak~Bakke Buan and Idun Reiten.
\newblock {F}rom tilted to cluster-tilted algebras of {D}ynkin type.
\newblock Preprint, math.RT/0510445v1.

\bibitem[BRS]{BRS}
Aslak~Bakke Buan, Idun Reiten, and Ahmet~I. Seven.
\newblock Tame concealed algebras and cluster quivers of minimal infinite type.
\newblock {\em J. Pure Appl. Algebra}, 211(1):71--82, 2007.

\bibitem[BRT]{BRT}
 Aslak~Bakke Buan, Idun Reiten and Hugh Thomas.
\newblock Three kinds of mutations.
\newblock Preprint, arXiv:1005.0276. 

\bibitem[CB]{CB}
William Crawley-Boevey.
\newblock Exceptional sequences of representations of quivers.
\newblock {\em Carleton-Ottawa Math. Lecture Note Ser.}, 14 (1992).

\bibitem[CCS]{CCS}
Philippe Caldero, Fr{\'e}d{\'e}ric Chapoton, and Ralf Schiffler.
\newblock Quivers with relations and cluster tilted algebras.
\newblock {\em Algebr. Represent. Theory}, 9(4):359--376, 2006.

\bibitem[CK]{CK}
Xiao-Wu Chen and Henning Krause.
\newblock Introduction to coherent sheaves on weighted projective lines.
\newblock Preprint, arXiv:0911.4473

\bibitem[FZ1]{FZ1} S. Fomin, A. Zelevisky.
\newblock Cluster algebras I: foundations.
\newblock {\em J. Amer. Math. Soc.} 15 (2) (2002), 497--529.

\bibitem[FZ2]{FZ2} S. Fomin, A. Zelevisky.
\newblock Cluster algebras II: finite type classification.
\newblock {\em Invent. Math.} 154 (1) (2003), 63--121.

\bibitem[GL]{GL} W. Geigle, H. Lenzing.
\newblock Perpendicular categories with applications to representations and sheaves.
\newblock {\em J. Algebra} 144 (2) (1991), 273--343.

\bibitem[H1]{Happel1} D. Happel.
\newblock Quasitilted algebras.
\newblock {\em CMS Conf. Proc.} 23 (1998), 55--82.

\bibitem[H2]{Happel2} D. Happel.
\newblock A characterization of hereditary categories
with tilting object.
\newblock {\em Invent. Math.} 144 (2001), 381--398.

\bibitem[HR] {HR} D. Happel, Claus M. Ringel.
\newblock Tilted algebras.
\newblock {\em Trans. Amer. Math. Soc.} 274, (1982), 399 -- 443.


\bibitem[HRS]{HRS} D. Happel, I. Reiten and S. O. Smal\o.
\newblock { Tilting in abelian categories and quasitilted algebras}, {\em Mem. Amer. Math. Soc.} 575 (1996).

\bibitem[HU]{HU} Dieter Happel and Luise Unger.
\newblock On the set of tilting objects in hereditary categories. Representations of algebras and related topics.
\newblock {Amer. Math. Soc.}, {\em Fields Inst. Commun.} 45 (2005), 141--159.

\bibitem[H{\"u}1]{Hubner1} Thomas H\"ubner.
\newblock Exzeptionelle Vektorb\"undel und Reflektionen an Kippgarben \"uber Projektiven Gewichteten Kurven.
\newblock {\em Dissertation zur Erlangung des Doktorgrades des Fachbereichs Mathematik-Informatik der Universit\"at-Gesamthochschule Paderborn}, 1996.

\bibitem[H{\"u}2]{Hubner2} Thomas H\"ubner.
\newblock Reflections and almost concealed canonical algebras.
\newblock {\em Sonderforschungsbereich 343, Diskrete Strukturen in der Mathematik} Universit\"at Bielefeld, 1997.

\bibitem[H{\"u}3]{Hubner3} Thomas H\"ubner.
\newblock Rank additivity for Quasi-tilted algebras of canonical type.
\newblock {\em Colloquium Mathematicum}, 75 (1998), no. 2.

\bibitem[K]{Keller} B. Keller.
\newblock On Triangulated orbit categories.
\newblock {\em Doc. Math.} 10 (2005), 551--581.

\bibitem[L1]{Lenzing1} H. Lenzing.
\newblock Hereditary noetherian categories with a tilting complex.
\newblock {\em Proc. Amer. Math. Soc.} 125 (1997), no. 7, 1893 -- 1901.

\bibitem[L2]{Lenzing2} H. Lenzing.
\newblock Hereditary  categories.
\newblock Handbook of tilting theory.
\newblock {\em London Math. Soc. Lecture Note Ser.} 332 (2007), 105--146.

\bibitem[LM1]{LM1}
Helmut Lenzing and Hagen Meltzer.
\newblock Tilting sheaves and concealed-canonical algebras.
\newblock {\em CMS Conf. Proc.} 18 (1996), 455--473.

\bibitem[LM2]{LM3}
Helmut Lenzing and Hagen Meltzer.
\newblock Exceptional sequences determined by their Cartan matrix.
\newblock {\em Algebr. Represent. Theory} 5 (2002), no. 2, 201--209.

\bibitem[LP]{LP}
Helmut Lenzing and J.A. de la Pe{\~n}a.
\newblock Wild canonical algebras.
\newblock {\em Math. Z.} 224 (2007), no. 3, 403--425.

\bibitem[M1]{Meltzer}
Hagen Meltzer.
\newblock Exceptional sequences for canonical algebras.
\newblock {\em Arch. Math. (Basel)} 64 (1995), no. 4, 304--312.

\bibitem[M2]{Meltzer2}
Hagen Meltzer.
\newblock Exceptional vector bundles, tilting sheaves and tilting complexes for weighted projective lines.
\newblock {\em Mem. Amer. Math. Soc.} 171 (2004), no. 808.

\bibitem[RS]{RiedtmannSchofield}
Christine Riedtmann and Aidan Schofield.
\newblock On a simplicial complex associated with tilting modules.
\newblock {\em Comment. Math. Helv.} 66 (1991), no.~1, 70--78.
 
\bibitem[R1]{Ringel1}
Claus~Michael Ringel.
\newblock Representations of K-species and bimodules.
\newblock {\em J. Algebra} 41 (1976), 269--302.

\bibitem[R2]{Ringel3}
Claus~Michael Ringel.
\newblock The braid group action on the set of exceptional sequences of a hereditary {A}rtin algebra.
\newblock Abelian group theory and related topics ({O}berwolfach, 1993), {\em Contemp. Math.} 171 (1994), 339--352.

\bibitem[R3]{Ringel4}
Claus~Michael Ringel.
\newblock Appendix: Some remarks concerning tilting modules and tilted algebras. Origin. Relevance. Future.
\newblock {\em Handbook of tilting theory, London Math. Soc. Lecture Note Ser.}, 332 (2007), 413--472.

\bibitem[V]{Dagfinn}
Dagfinn Vatne.
\newblock The mutation class of $D_n$ quivers.
\newblock{\em Comm. Algebra} 38 (2010), issue 3, 1137--1146.

\bibitem[Z]{Z}
Bin Zhu.
\newblock Equivalences between cluster categories.
\newblock {\em J. Algebra}, 304(2):832--850, 2006.

\end{thebibliography}
\end{document}